\documentclass[10pt,a4paper]{amsart}
\usepackage[utf8]{inputenc}
\usepackage{amsmath, amsthm}
\usepackage{mathtools}
\usepackage{amsfonts}
\usepackage{amssymb}
\usepackage{multirow}
\usepackage{graphicx}
\usepackage{xurl}
\usepackage{float}

\usepackage{natbib}

\newtheorem{lem}{Lemma}
\newtheorem{prop}{Proposition}
\newtheorem{cor}{Corollary}
\newtheorem{deff}{Definition}
\newtheorem{thm}{Theorem}

\newcommand{\N}{\mathbb{N}}
\newcommand{\E}{\mathbb{E}}
\newcommand{\Prob}{\mathbb{P}}
\newcommand{\Var}{\operatorname{Var}}
\newcommand{\Z}{\mathbb{Z}}
\newcommand{\Xn}{(X_n)_{n\in\mathbb{Z}}}

\newcommand{\wur}{^{\frac{1}{2}}}
\newcommand{\del}{^{\frac{\delta}{2(2+\delta)}}}
\newcommand{\bm}{\beta_{m-2l}^{\frac{\delta}{2+\delta}}}

\begin{document}
\bibliographystyle{plainnat}

\title[Robust Change-Point Detection for Functional Time Series]{Robust Change-Point Detection for Functional Time Series Based on $U$-Statistics and Dependent Wild Bootstrap}

\author{Lea Wegner}

\date{\today}

\author{Martin Wendler}
\address{Otto-von-Guericke-Universit\"at Magdeburg, Germany}
\email{martin.wendler@ovgu.de}

\thanks{The research was supported by the German Research Foundation (Deutsche Forschungsgemeinschaft – DFG), project WE 5988/3 \emph{Analyse funktionaler Daten ohne Dimensionsreduktion}. We thank Claudia Kirch for fruitfull discussions on the topic. We are gratefule for the  useful and detailed comments two anonymous refrees have provided.}

\begin{abstract} The aim of this paper is to develop a change-point test for functional time series that uses the full functional information and is less sensitive to outliers compared to the classical CUSUM test. For this aim, the Wilcoxon two-sample test is generalized to functional data. To obtain the asymptotic distribution of the test statistic, we prove a limit theorem for a process of $U$-statistics with values in a Hilbert space under weak dependence. Critical values can be obtained by a newly developed version of the dependent wild bootstrap for non-degenerate 2-sample $U$-statistics.

\end{abstract}


\maketitle

\thispagestyle{empty}

\section{Introduction}

Statistical methods for observations consisting of functions are widely discussed since at least the work by \cite{ramsay1982data}, and there is a growing interest in recent years because more and more data is available in high resolution that can not be treated as multivariate data. Functional data analysis might even be helpful for one-dimensional time series (see e.g. \cite{hoerkok09}). Functional observations are often modelled as random variables taking values in a Hilbert space, we recommend the book by \cite{hormann2012functional} for an introduction.

In this paper, we will propose new methods for the detection of change-points: Suppose that we observe $X_1,...,X_n$ being a part of a time series  $\Xn$  with values in a separable Hilbert space $H$ (equipped with inner product $\langle \cdot,\cdot, \rangle$ and norm $\|\cdot\|=\sqrt{\langle \cdot,\cdot \rangle}$). The \emph{at most one change-point problem} is to test the null hypothesis of stationarity against the alternative of an abrupt change of the distribution at an unknown time point $k^\star$: $X_1\stackrel{\mathcal{D}}{=}...\stackrel{\mathcal{D}}{=}X_{k^\star}$ and $X_{k^\star+1}\stackrel{\mathcal{D}}{=}...\stackrel{\mathcal{D}}{=}X_{n}$, but $X_1\stackrel{\mathcal{D}}{\neq}X_n$ (where $X_i\stackrel{\mathcal{D}}{=}X_j$ means that $X_i$ and $X_j$ have the same distribution).

Functional data is often projected on lower dimensional spaces with functional principal components, see \cite{berkes2009detecting} for a change in mean of independent data and \cite{aston2012detecting} for a change in mean of time series. \cite{fremdt2014functional} proposed to let the dimension on the subspace on which the data is projected grow with the sample size. But is is also possible to use change-point tests without dimension reduction as done by \cite{horvath2014testing} under independence, by \cite{STW16} and \cite{aue2018detecting} under dependence. Since using the asymptotic distribution would require knowledge of the infinite-dimensional covariance operator, it is convenient to use bootstrap methods. In the context of change-point detection for functional time series,  the nonoverlapping block bootstrap was studied by \cite{STW16}, the dependent wild bootstrap by \cite{bucchia2017change} and the block multiplier bootstrap (for Banach-space-valued times series) by \cite{dette2020functional}.

Typically, these tests are based on variants of the CUSUM-test, where CUSUM stands for cumulated sums. Such tests  make use of sample means and thus, they are sensitive to outliers. For real-valued time series, several authors have constructed more robust tests  based on the Mann–Whitney–Wilcoxon-$U$-test. For the two-sample problem (do the two real-valued samples $X_1,...,X_{n_1}$ and $Y_1,...,Y_{n_2}$ have the same location?), the Mann–Whitney–Wilcoxon-$U$-statistic can be written as
\begin{equation*}
U(X_1,...,X_{n_1},Y_1,...,Y_{n_2})=\frac{1}{n_1n_2}\sum_{i=1}^{n_1}\sum_{j=1}^{n_2}\operatorname{sgn}(X_i-Y_j)=\frac{1}{n_1n_2}\sum_{i=1}^{n_1}\sum_{j=1}^{n_2}\frac{X_i-Y_j}{|X_i-Y_j|}
\end{equation*}
(where $0/0$ is set to $0$). \cite{CC17} have generalized this test statistic to Hilbert spaces by replacing the sign by the so called spatial sign:
\begin{equation*}
U(X_1,...,X_{n_1},Y_1,...,Y_{n_2})=\frac{1}{n_1n_2}\sum_{i=1}^{n_1}\sum_{j=1}^{n_2}\frac{X_i-Y_j}{\|X_i-Y_j\|}
\end{equation*}
They have shown the weak convergence to a Gaussian distribution for independent random variables. For change-point detection, one encounters several problems: In practice, the change-point is typically unknown, so it is not known where to split the sequence of the observations into two samples. In many applications, the assumption of independence is not realistic, one rather has to deal with time series. Furthermore, the covariance operator is not known.

To deal with these problems, we will study limit theorems for two-sample $U$-processes with values in Hilbert spaces and deduce the asymptotic distribution of the Wilcoxon-type change-point-statistic
\begin{equation*}
\max_{k=1,...,n-1}\Big\|\frac{1}{n^{3/2}}\sum_{i=1}^k\sum_{j=k+1}^n\frac{X_i-X_j}{\|X_i-X_j\|}\Big\|
\end{equation*}
for a short-range dependent, Hilbert-space-valued time series $\Xn$. Change-point tests based on Wilcoxon have been studied before, but mainly for real-valued observations, starting with \cite{darkhovsky1976non} and \cite{pettitt1979non}. \cite{yu2022robust} used the maximum of componentwise Wilcoxon-type statistics. Very recently and independently of our work, \cite{jiang2022robust} introduced a test statistic based on spatial signs for independent, high-dimensional observations, which is very similar to the square of our test statistic. However, \cite{jiang2022robust} obtained the limit for a growing dimension of the observations and assuming that the entries of each vector form a stationary, weakly dependent time series, while we consider observations in a fixed Hilbert space $H$ and take the limit for a growing number of observations. Furthermore, they use self-normalization instead of bootstrap to obtain critical values.

Let us note that spatial signs have been used for change-point detection before by other authors: \cite{vogel2015robust} have studied a robust test for changes in the dependence structure of a finite-dimensional time series based on the spatial sign covariance matrix.

As the Mann–Whitney–Wilcoxon-$U$-statistic is a special case of a two-sample $U$-statistic, authors like \cite{csoehorinvpr}, \cite{gombay2002rates} studied more general $U$-statistics for change point detection under independence and \cite{DFGW13}  under dependence. We will provide our theory not only for the special case of the test statistic based on spatial signs, but for general test statistics based on two-sample $H$-valued $U$-statistics under dependence.

As the limit depends on the unknown, infinite-dimensional long-run covariance operator, one would either need to estimate this operator, or one could use resampling techniques. \cite{LN13} have developed a variant of the dependent wild bootstrap (introduced by \cite{S10}) for $U$-statistics. However, their method works only for degenerate $U$-statistics. As the Wilcoxon-type statistic is non-degenerate, we propose a new version of the dependent wild bootstrap for this type of $U$-statistic. The bootstrap version of our change-point test statistic is
\begin{equation*}
\max_{k=1,...,n-1}\Big\|\frac{1}{n^{3/2}}\sum_{i=1}^k\sum_{j=k+1}^n\frac{X_i-X_j}{\|X_i-X_j\|}(\varepsilon_i+\varepsilon_j)\Big\|,
\end{equation*}
where $\varepsilon_1,...,\varepsilon_n$ is a stationary sequence of dependent $N(0,1)$-distributed multipliers, independent of $X_1,...,X_n$. We will prove the asymptotic validity of our new bootstrap method. Our variant of the dependent wild bootstrap is similar, but not identical to the variant proposed by \cite{doukhan2015dependent} for non-degenerate von Mises statistics. Note that this bootstrap differs from the multiplier bootstrap proposed by \cite{bucher2016dependent}, as it does not rely on pre-linearization, that means replacing the $U$-statistic by a partial sum.

\section{Main Results}

We will treat the CUSUM statistic and the Wilcoxon-type statistic as two special cases of a general class based on two-sample $U$-statistics. Let $h:H^2\rightarrow H$ be a kernel function. We define
\begin{equation*}
U_{n,k}=\sum_{i=1}^k\sum_{j=k+1}^nh(X_i,X_j).
\end{equation*}
For $h(x,y)=x-y$, we obtain with a short calculation
\begin{equation*}
\max_{1\leq k < n}\frac{1}{n^{3/2}}\left\|U_{n,k}\right\|=\max_{1\leq k < n}\frac{1}{\sqrt{n}}\Big\|\sum_{i=1}^k\big(X_i-\frac{1}{n}\sum_{j=1}^n X_j\big)\Big\|,
\end{equation*}
which is the CUSUM-statistic for functional data. On the other hand, with the kernel $h(x,y)=(x-y)/\|x-y\|$, we get the Wilcoxon-type statistic. Other kernels would be possible, e.g. $h(x,y)=(x-y)/(c+\|x-y\|)$ for some $c>0$ as a compromise between the CUSUM and the Wilcoxon approach. Before stating our limit theorem for this class based on two-sample $U$-statistics, we have to define some concepts and our assumptions.

We will start with our concept of short range dependence, which is based on a combination of absolute regularity (introduced by \cite{volkonskii1959some}) and $P$-near-epoch dependence (introduced by \cite{DVWW17}). In the following, let $H$ be a separable Hilbert space with inner product $\langle \cdot,\cdot \rangle$ and norm $\|x\|=\sqrt{\langle x,x\rangle}$.

\begin{deff}[Absolute Regularity] Let $(\zeta_n)_{n\in\mathbb{Z}}$ be a stationary sequence of random variables. We define the mixing coefficients $(\beta_m)_{m\in\Z}$ by
\begin{equation*}
\beta_m=E\Big[\sup_{A\in\mathcal{F}_{m}^\infty}\left(P(A|\mathcal{F}_{-\infty}^{0})-P(A)\right)\Big],
\end{equation*}
where $\mathcal{F}_{a}^b$ is the $\sigma$-field generated by $\zeta_{a},\ldots,\zeta_b$, and call the sequence $(\zeta_n)_{n\in\Z}$ absolutely regular if $\beta_m\rightarrow 0$ as $m\rightarrow\infty$. 
\end{deff}

\begin{deff}[P-NED]
Let $(\zeta_n)_{n\in\mathbb{Z}}$ be a stationary sequence of random variables. $(X_n)_{n\in\mathbb{Z}}$ is called near-epoch-dependent in probability (P-NED) on $(\zeta_n)_{n\in\mathbb{Z}}$  if there exist sequences $(a_k)_{k\in\mathbb{N}}$ with $a_k \xrightarrow{k \to \infty} 0$ and $(f_k)_{k\in\mathbb{Z}}$ and a nonincreasing function $\Phi:(0,\infty) \rightarrow (0,\infty)$ such that 
\[ \mathbb{P}(\Vert X_0-f_k(\zeta_{-k},...,\zeta_k)\Vert > \epsilon) \leq a_k\Phi(\epsilon) \,\,\, \forall k\in \mathbb{N},\, \epsilon >0 .\]
\end{deff}

\begin{deff}[$L_p$-NED]
Let $(\zeta_n)_{n\in\mathbb{Z}}$ be a stationary sequence of random variables. $(X_n)_{n\in\mathbb{Z}}$ is called $L_p$-NED on $(\zeta_n)_{n\in\mathbb{Z}}$ if there exists a sequence of approximation constants $(a_k)_{k\in\mathbb{N}}$ with $a_k \xrightarrow{k\to\infty} 0$ and 
\[ \mathbb{E}[\Vert X_0-\mathbb{E}[X_0| \mathfrak{F}_{-k}^k] \Vert^p  ]^{\frac{1}{p}} \leq a_{k,p} .\] 
\end{deff}
P-NED has the advantage of not implying finite moments (unlike $L_p$-NED), which is useful to allow for heavy tailed distributions.

Additionally, we will need assumptions on the kernel:
\begin{deff}[Antisymmetry] A kernel $h:H^2\rightarrow H$ is called antisymmetric, if for all $x,y\in H$
\begin{equation*}
h(x,y)=-h(y,x).
\end{equation*}
\end{deff}
Antisymmetric kernels are natural candidates for comparing two distributions, because if $X$ and $\tilde{X}$ are independent, $H$-valued random variables with the same distribution and $h$ is antisymmetric, we have $E[h(X,\tilde{X})]=0$, so our test statistic should have values close to 0, see also \cite{ravckauskas2020convergence}. 
\begin{deff}[Uniform Moments] If  there is a $M>0$ such that for all $k,n \in\mathbb{N}$
\[\mathbb{E}[\Vert h\big(f_k(\zeta_{-k},...,\zeta_k),f_k(\zeta_{n-k},...,\zeta_{n+k})\big)\Vert_{{H}}^{m}] \leq M,  \]
\[ \mathbb{E}[\Vert h\big(X_0,f_k(\zeta_{n-k},...,\zeta_{n+k})\big)\Vert_{{H}}^{m}] \leq M, \]
\[ \mathbb{E}[\Vert h\big(X_0,X_n\big)\Vert_{{H}}^{m}] \leq M,  \] 
we say that the kernel has uniform $m$-th moments under approximation.
\end{deff}
Furthermore, we need the following mild continuity condition on the kernel, which is called variation condition and was introduced by  \cite{denker1986rigorous}. The kernel $h(x,y)=(x-y)/\|x-y\|$ will fulfill the condition, as long as there exists a constant $C$ such that $P(\|X_1-x\|\leq \epsilon)\leq C\epsilon$ for all $x\in H$ and $\epsilon>0$. This can be proved along the lines of Remark 2 in \cite{dehling2022change}.  $P(\|X_1-x\|\leq \epsilon)\leq C\epsilon$ for all $x\in H$, $\epsilon>0$ does not hold if the distribution of $X_1$ has points with positive mass, but it still can hold if  the distribution is concentrated on finite-dimensional sub-spaces.

\begin{deff}[Variation condition] The kernel $h$ fulfills the variation condition if there exist $L$, $\epsilon_0 > 0$ such that for every $\epsilon \in (0, \epsilon_0)$:
\[ \mathbb{E}\bigg[\Big( \sup_{\substack{\Vert x-X\Vert \leq \epsilon \\ \Vert y-\tilde{X}\Vert \leq \epsilon }} \Vert h(x,y)-h(X,\tilde{X})\Vert_{{H}} \Big)^2  \bigg] \leq L\epsilon \]
\end{deff}
Finally, we will need Hoeffding's decomposition of the kernel to be able to define the limit distribution:
\begin{deff}[Hoeffding's decomposition]
Let $h:H\times H \rightarrow {H}$ be an antisymmetric kernel. Let $X,\tilde{X}$ be two i.i.d. random variables with the same distribution as $X_1$. Hoeffding's decomposition of $h$ is defined as
\[ h(x,y)= h_1(x)-h_1(y)+h_2(x,y) \, \forall x,y \in H  \]
where \[h_1(x) = \E[h(x,\tilde{X})]  \]
\[h_2(x,y)= h(x,y) - \E[h(x,\tilde{X})] - \E[h(X,y)] = h(x,y) -h_1(x)+h_1(y)  \]
\end{deff}

Now we can state our first theorem on the asymptotic distribution of our test statistic under the null hypothesis (stationarity of the time series):
\begin{thm}\label{theo1}
Let $(X_n)_{n\in\mathbb{Z}}$ be stationary and P-NED on an absolutely regular sequence $(\zeta_n)_{n\in\mathbb{Z}}$ such that $a_k \Phi(k^{-8\frac{\delta+3}{\delta}})= \mathcal{O}(k^{-8\frac{(\delta+3)(\delta+2)}{\delta^2}})$ and $\sum_{k=1}^\infty k^2 \beta_k^{\frac{\delta}{4+\delta}} < \infty$ for some $\delta>0$. Assume that $h:H^2\rightarrow H$ is an antisymmetric kernel that fulfills the variation condition and is either bounded or has uniform $(4+\delta)$-moments under approximation.
Then it holds that 
\[ \max_{1\leq k<n} \frac{1}{n^{3/2}} \Big\| \sum_{i=1}^k\sum_{j=k+1}^n h(X_i,X_j) \Big\| \xrightarrow{\mathcal{D}} \sup_{\lambda\in[0,1]} \Vert W(\lambda)-\lambda W(1) \Vert   \]
where $W$ is an $H$-valued Brownian motion and the covariance operator $S$ of $W(1)$ is given by
\begin{equation*}
\langle S(x),y\rangle =\sum_{i=-\infty}^\infty \operatorname{Cov}\left(\langle h_1(X_0),x\rangle,\langle h_1(X_i),y\rangle\right).
\end{equation*}
\end{thm}
For the kernel $h(x,y)=x-y$, we obtain as a special case a limit theorem for the functional CUSUM-statistic similar to Corollary 1 of \cite{STW16} (although our assumptions on near epoch dependence are stronger). In the next section, we will compare the Wilcoxon-type statistic and the CUSUM-statistic with a simulation study. The proofs of the results can be found in Section \ref{Proof Main}. The next theorem will show that the test statistic converges to infinity in probability under some alternatives, so a test based on this statistic consistently detects these type of changes.

For this, we consider the following model: We have a stationary, $H\otimes H$-valued sequence $(X_n, Z_n)_{n\in\Z}$ and we observe $Y_1,...,Y_n$ with
\begin{equation*}
Y_i=\begin{cases}X_i \  \ &\text{for}\ i\leq \lfloor n\lambda^\star\rfloor = k^\star\\Z_i \  \ &\text{for}\ i> \lfloor n\lambda^\star \rfloor = k^\star\end{cases},
\end{equation*}
so $\lambda^\star\in(0,1)$ is the proportion of observations after which the change happens. If the distribution of $X_i$ and $Z_i$ is not the same, then the alternative hypothesis holds: $X_1\stackrel{\mathcal{D}}{=}...\stackrel{\mathcal{D}}{=}X_{k^\star}$ and $X_{k^\star+1}\stackrel{\mathcal{D}}{=}...\stackrel{\mathcal{D}}{=}X_{n}$, but $X_1\stackrel{\mathcal{D}}{\neq}X_n$. A simple example might be $Z_i=X_i+\mu$, where $\mu\in H$ and $\mu\neq 0$. However, let us point out that not all changes in distribution can be consistently detected. The change is detectable, if $E[h(X_1,\tilde{Z}_1)]\neq 0$ for an independent copy $\tilde{Z}_1$ of ${Z}_1$. For example, with the kernel $h(x,y)=x-y$ and $Z_i=X_i+\mu$ with $\mu\neq 0$, the change is always detectable.

\begin{thm}\label{theo2}Let $(X_n, Z_n)_{n\in\Z}$ be P-NED on an absolutely regular sequence $(\zeta_n)_{n\in\mathbb{Z}}$ such that $a_k \Phi(k^{-8\frac{\delta+3}{\delta}})= \mathcal{O}(k^{-8\frac{(\delta+3)(\delta+2)}{\delta^2}})$ and $\sum_{k=1}^\infty k^2 \beta_k^{\frac{\delta}{4+\delta}} < \infty$ for some $\delta>0$. Assume that $h:H^2\rightarrow H$ is an antisymmetric kernel that fulfills the variation condition and is either bounded or has uniform $(4+\delta)$-moments under approximation for both processes $(X_n)_{n\in\Z}$ and $(Z_n)_{n\in\Z}$, that $E[\Vert h(X_1,\tilde{Z}_1)\Vert^{4+\delta}]<\infty$, and that  $E[h(X_1,\tilde{Z}_1)]\neq 0$, were $\tilde{Z}_1$ is an independent copy of ${Z}_1$. Then 
\[ \max_{1\leq k<n} \frac{1}{n^{3/2}} \Big\|\sum_{i=1}^k\sum_{j=k+1}^n h(Y_i,Y_j) \Big\| \xrightarrow{\mathcal{P}} \infty . \]
\end{thm}
These results on the asymptotic distribution can not be applied directly in many practical applications, because the covariance operator is unknown. For this reason, we introduce the dependent wild bootstrap for non-degenerate $U$-statistics: Let $(\varepsilon_{i,n})_{i\leq n, n\in\N}$ be a rowwise stationary triangular scheme of $N(0,1)$-distributed variables (we often drop the second index for notational convenience: $\varepsilon_{i}=\varepsilon_{i,n}$). The bootstrap version of our $U$-statistic is then
\begin{equation*}
U_{n,k}^\star=\sum_{i=1}^k\sum_{j=k+1}^nh(X_i,X_j)(\varepsilon_i+\varepsilon_j).
\end{equation*}
\begin{thm}\label{theo3}Let the assumptions of Theorem \ref{theo1} hold for $(X_n)_{n\in\Z}$ and $h:H^2\rightarrow H$. Assume that  $(\varepsilon_{i,n})_{i\leq n, n\in\N}$ is independent of $(X_n)_{n\in\Z}$, has standard normal marginal distribution and $\operatorname{Cov}(\varepsilon_i,\varepsilon_j)=w(|i-j|/q_n)$, where $w$ is symmetric and continuous  with $w(0)=1$ and $\int_{-\infty}^\infty |w(t)|dt<\infty$. Assume that $q_n\rightarrow\infty$ and $q_n/n\rightarrow0$. Then it holds that 
\begin{multline*}\left(\max_{1\leq k<n} \frac{1}{n^{3/2}}\Big\| U_{n,k}\Big\|, \max_{1\leq k<n} \frac{1}{n^{3/2}}\Big\| U_{n,k}^\star\Big\|\right)\\
  \xrightarrow{\mathcal{D}} \bigg(\sup_{\lambda\in[0,1]} \Vert W(\lambda)-\lambda W(1) \Vert,\sup_{\lambda\in[0,1]} \Vert W^\star(\lambda)-\lambda W^\star(1) \Vert\bigg)   
\end{multline*}
where  $W$ and $W^\star$ are two independent, $H$-valued Brownian motions with covariance operator as in Theorem \ref{theo1}.
\end{thm}
From this statement, it follows that the bootstrap is consistent and it can be evaluated using the Monte Carlo method. If you generate several copies of the bootstraped test statistic independent conditional on $X_1,..,X_n$, the empirical quantiles of the bootstraped test statistics can be used as critical values for the test. For a deeper discussion on bootstrap validity, see \cite{bucher2019note}. Of course, in practical applications, the function $w$ and the bandwidth $q_n$ have to be chosen. We will apply a method by \cite{RS17} for the bandwidth selection.

Instead of using multipliers with a standard normal distribution, one might also choose other distributions for $(\varepsilon_{i,n})_{i\leq n, n\in\N}$. This is done for the traditional wild bootstrap to capture skewness. Under the hypothesis, the distribution of $h(X_i,X_j)$ is close to symmetric for $i$ and $j$ far apart, so we do not expect a large improvement by non-Gaussian multipliers and limit our analysis in this paper to the case of Gaussian multipliers.

\section{Data Example and Simulation Results} \label{Simulations}

\subsection*{Bootstrap procedure}

Since no theoretical values of the limit distribution of our test-statistic exist, we perform a bootstrap to find critical values for a test-decision. The procedure to find the critical value for significance level $\alpha \in (0,1)$ is the following:
\begin{itemize}
\item Calculate $h(X_i,X_j)$ for all $i<j$
\item For each of the bootstrap iterations $t=1,...,m$:
\begin{itemize}
\item Calculate $h(X_i,X_j)(\varepsilon_i^{(t)}+\varepsilon^{(t)}_j)$, where $(\varepsilon^{(t)}_i)_{i<n}$ are random multiplier
\item Calculate $U_{n,k}^{(t)}=\sum_{i=1}^k\sum_{j=k+1}^n h(X_i,X_j)(\varepsilon^{(t)}_i+\varepsilon^{(t)}_j)$ for all $k < n $
\item Find $\max\limits_{1 \leq k<n} \Vert U_{n,k}^{(t)}\Vert$ 
\end{itemize}
\item Identify the empirical $\alpha$-quantile $U_\alpha$ of all $\max\limits_{1 \leq k<n} \Vert U_{n,k}^{(1)} \Vert,...,\max\limits_{1 \leq k<n} \Vert U_{n,k}^{(m)} \Vert$
\item Calculate $U_{n,k}=\sum_{i=1}^k\sum_{j=k+1}^n h(X_i,X_j)$ for all $1\leq k<n$
\item Test decision: If $\max\limits_{1\leq k<n} \Vert U_{n,k} \Vert > U_{\alpha}$, reject the null hypothesis 
\end{itemize}  
To ensure a certain covariance structure within the multiplier (that fulfills the assumptions of the multiplier theorem), we calculate them as 
\[(\varepsilon^{(t)}_i)_{i\leq n}= A (\eta_i)_{i\leq n}   \]
where $\eta_1,...,\eta_i$ are i.i.d. $N(0,1)$-distributed and $A$ is the square root of the quadratic spectral covariance matrix constructed with bandwidth-parameter $q$ (chosen with the method by \cite{RS17} described below). That means $AA^t=B$, where $B$ has the entries 
\[ B_{i,j} = v_{\vert i-j \vert} \;\;\;\;\;\; \forall \, 1 \leq i,j \leq n    \]
with 
\begin{align*}
& v_0 = 1 \\
& v_i = \frac{25}{12 \pi^2 (i-1)^2/q^2} \left( \frac{\sin(\frac{6\pi(i-1)/q}{5})}{\frac{6\pi(i-1)/q}{5}} -\cos(\frac{6\pi(i-1)/q}{5})\right) \;\;\;\;\;\; \forall \, 1 \leq i \leq n-1.
\end{align*}

\subsection*{Bandwidth}
We use a data adapted bandwidth parameter $q_{adpt}$ in the bootstrap which is evaluated for each simulated data sample $X_1,...,X_n$ by the following procedure:
\begin{itemize}
\item Calculate $\tilde{X}_1,...,\tilde{X}_n$ where $\tilde{X}_i=\frac{1}{n-1}\sum_{j=1, j\neq i}^n h(X_i,X_j)$
\item Determine a starting value $q_0=n^{1/5}$
\item Calculate matrices $V_k = \frac{1}{n}\sum_{i=1}^{n-(k-1)} \tilde{X_i} \otimes \tilde{X_k} $ for $k=1,..., q_0$,  where $\otimes$ is the outer product
\item Compute $CP_0=V_1+2\sum_{k=1}^{q_0-1}w(k, q_0)V_{k+1}$ \\
 \hspace*{20pt}   and $CP_1= 2\sum_{k=1}^{q_0-1}k \,w(k, q_0)V_{k+1}$ \\
 $w$ is a kernel function, we use the quadratic spectral kernel \\
 $w(k,q)= \frac{25}{12 \pi^2 k^2/q^2} \left( \frac{\sin(\frac{6\pi k/q}{5})}{\frac{6\pi k/q}{5}} -\cos(\frac{6\pi k/q}{5})\right) $
 \item Receive the data adapted bandwidth 
 \[ q_{adpt} = \left \lceil \left( \frac{3n\sum_{i=1}^d \sum_{j=1}^d {CP_1}_{i,j}}{\sum_{i=1}^d\sum_{j=1}^d{CP_0}_{i,j}+ \sum_{j=1}^d {CP_0}_{j,j}^2 }  \right)^{1/5}    \right \rceil    \]
\end{itemize}

For theoretical details about the data adapted bandwidth we refer to \cite{RS17}.

\subsection*{Data example}
We look at data of $344$ monitoring stations of the 'Umweltbundesamt' for air pollutants located all over Germany
(Source: Umweltbundesamt, \url{https://www.umweltbundesamt.de/daten/luft/luftdaten/stationen} Accessed on 06.08.2020). The particular data is the daily average of particulate matter with particles smaller than $10 \mu m$ ($PM_{10}$) measured in $\mu g / m^3$ from January 1, 2020 to May 31, 2020. This means we have $n=152$ observations and treat the measurements of all stations on one day as a data from $\mathbb{R}^{344}$.

Since the official restrictions of the German Government in course of the COVID-19 pandemic came into force on March 22, 2020, an often asked question was whether these restrictions (social distancing, closed gastronomy, closed/reduced work or work from home) had an effect on the air quality in Germany. This question comes from the assumption that the restrictions lead to reduced traffic, resulting in reduced amount of particulate matter.

There are several publications from various countries studying the effects of lockdown measures on air pollution parameters like nitrogen oxides ($NO$, $NO_2$), ozone ($O_3$) and particulate matter ($PM_{10}$, $PM_{2.5}$). For example, \cite{lian2020impact} investigated data from the city of Wuhan, or \cite{zangari2020air} for New York City. Data for Berlin, as for 19 other Cities around the world, are investigated by \cite{fu2020impact}. They observed a decline in particular matter ($PM_{10}$ and $PM_{2.5}$, only significant for $PM_{2.5}$) in the period of lockdown. But the observed time period is rather short (one month - Mar. 17 to Apr. 19, 2020) and the findings for a densely populated city may not simply be transferred to the whole of Germany. In contrast to that, we use data from measuring stations located across the whole country and over a period of five months.

Looking at the empirical p-values of the CUSUM test and the Wilcoxon-type test (based on spatial signs) resulting from $m=3000$ Bootstrap iterations in Table \ref{p-values}, we see that with CUSUM, the null hypothesis $H_0$ is never rejected for any significance level $\alpha < 0.2 $. But the Wilcoxon-type test rejects $H_0$ for significance level $\alpha$ larger than $0.03$. \\
\begin{table}[h]
\begin{center}
\begin{tabular}{ |c|c|c| } 
\hline
\multicolumn{2}{|c|}{p-values} \\
\hline
  CUSUM & Spatial Sign \\
\hline
 $0.226$ & $0.027$ \\
\hline
\end{tabular}
\vspace{3pt}
\caption{Empirical p-values for CUSUM and spatial sign test with data adapted bandwidth. $m=3000$ Bootstrap iterations were used.}
\label{p-values}
\end{center}
\end{table}
Since the data exhibits a massive outlier located at January 1 (likely due to New Year's firework), we repeated the test procedure without the data of this day. We observed that the resulting p-value for the Wilcoxon-type test changed just slightly (Table \ref{p-values_woNY}). Whereas the p-value for CUSUM decreased notably - it is now around $0.08$. In this example we see that CUSUM is clearly more influenced by the outlier in the data than the spatial signs based test. Evaluation showed that the data adapted bandwidth was set to $q_{adpt}=3$ for both the CUSUM test and the Wilcoxon-type test for both scenarios.\\ 
\begin{table}[h]
\begin{center}
\begin{tabular}{ |c|c|c| } 
\hline
\multicolumn{2}{|c|}{p-values (data excluding Jan. 1)} \\
\hline
 CUSUM & Spatial Sign \\
\hline
 $0.078$ & $0.030$ \\
\hline
\end{tabular}
\vspace{3pt}
\caption{Empirical p-values for CUSUM and spatial sign test with data adapted bandwidth for data excluding January 1, 2020. $m=3000$ Bootstrap iterations were used.}
\label{p-values_woNY}
\end{center}
\end{table}
\begin{center}
\begin{figure}[t]
\includegraphics[width=\textwidth]{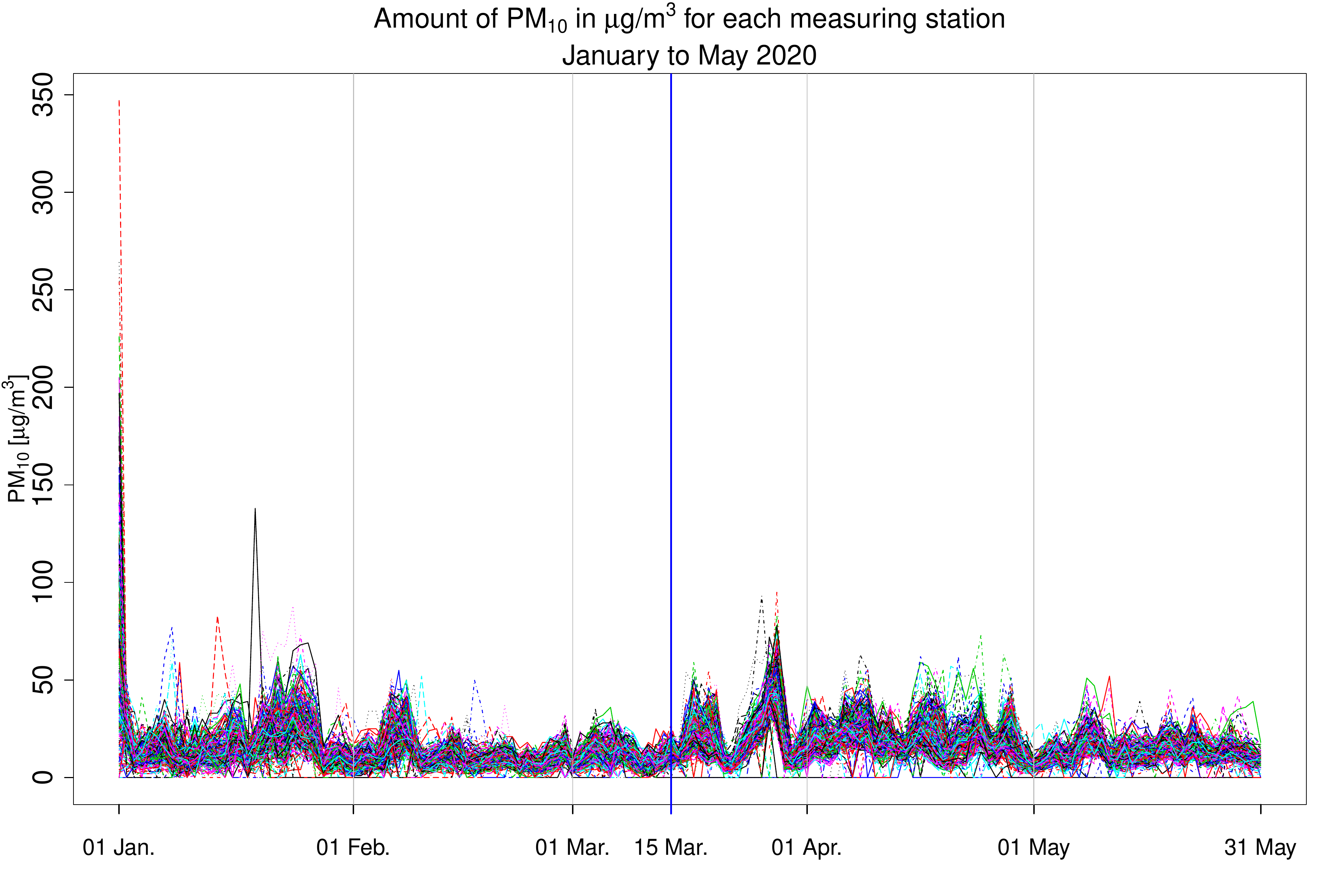} 
\caption{Daily average of $PM_{10}$ in $\mu g / m^3$ for $344$ monitoring stations from January 1, 2020 to May 31, 2020. Each line corresponds to one station. The blue vertical line is the estimated change-point location. The massive outlier at January 1 could result from New Year's fireworks.}
\label{Feinstaub}
\end{figure}
\end{center}
A natural approach to estimate the location $\hat{k}$ of the change-point, is to determine the smallest $1 \leq k<n$ for which the test statistic attains its maximum:
\[ \hat{k} = \min\{k: \Vert \frac{1}{n^{3/2}} U_{n,k} \Vert = \max_{1\leq j< n} \Vert \frac{1}{n^{3/2}} U_{n,j}  \Vert    \}   \] 
The maximum of the spatial sign test statistic, which marks our estimated change point, is received at March 15, 2020. (The maximum of the CUSUM statistic is indeed located at the same point.) The estimated change-point in our example lies a week before the official restrictions regarding COVID-19 were imposed. One could argue that the citizen, being aware of the situation, changed their behaviour beforehand, without strict official restrictions. Data projects using mobile phone data (e.g Covid-19 Mobility Project and Destatis) indeed show a decline in mobility preceding the official restrictions on March 22 by around a week. (see \url{https://www.covid-19-mobility.org/de/data-info/}, \url{https://www.destatis.de/DE/Service/EXDAT/Datensaetze/mobilitaetsindikatoren-mobilfunkdaten.html})

But if we look at our data (Fig. \ref{Feinstaub}), one gets the impression that a change in mean would rather be upwards than downwards, meaning that the daily average pollution increased after March 15, 2020 compared to the beginning of the year. Indeed, after averaging over the 344 monitoring stations and applying the two-sample Hodges-Lehmann estimator to the resulting one-dimensional time series, we estimate the average increase to be 3.8 $\mu g/m^3$. However, our test does not reject the null hypothesis when applied to this one-dimensional time series.

Similar findings about in increase in $PM_{10}$ were made by \cite{ropkins2021early}. They studied the impact of the COVID-19 lockdown on air quality across the UK. While using long-term data (Jan. 2015 to Jun. 2020) from Rural Background, Urban Background and Urban Traffic stations, they observed an increase for $PM_{10}$ and $PM_{2.5}$ while locking down. Noting that this trend is "highly inconsistent with an air quality response to the lockdown", they discussed the possibility that the lockdown did not greatly limit the largest impacts on particulate matter. We assume that the findings are to some extend comparable to Germany due to the similar geographic and demographic characteristics of the countries.

Furthermore, the German 'Umweltbundesamt' states that traffic is not the main contributor to $PM_{10}$ in Germany (anymore) and other sources of particulate matter (e.g. fertilization, Saharan dust, soil erosion, fires) can overlay effects of reduced traffic (source: \url{https://www.umweltbundesamt.de/faq-auswirkungen-der-corona-krise-auf-die#welche-auswirkungen-hat-die-corona-krise-auf-die-feinstaub-pm10-belastung}).
It is known that one mayor meteorological effect on particulate matter is precipitation, since it washes the dust out of the air (scavenging). Comparing the data with the meteorological recordings (Fig. \ref{Niederschlag}) another explanation for the change-point gets visible: While January was relatively warm with few precipitation, February and first half of March had much of it. Beginning in the middle of March, a relatively drought period started and lasted through April and May. (Data extracted from DWD Climate Data Center (CDC): Daily station observations precipitation height in mm, v19.3, 02.09.2020.  \url{https://cdc.dwd.de/portal/202107291811/mapview})

\begin{center}
\begin{figure}[h]
\includegraphics[width=\textwidth]{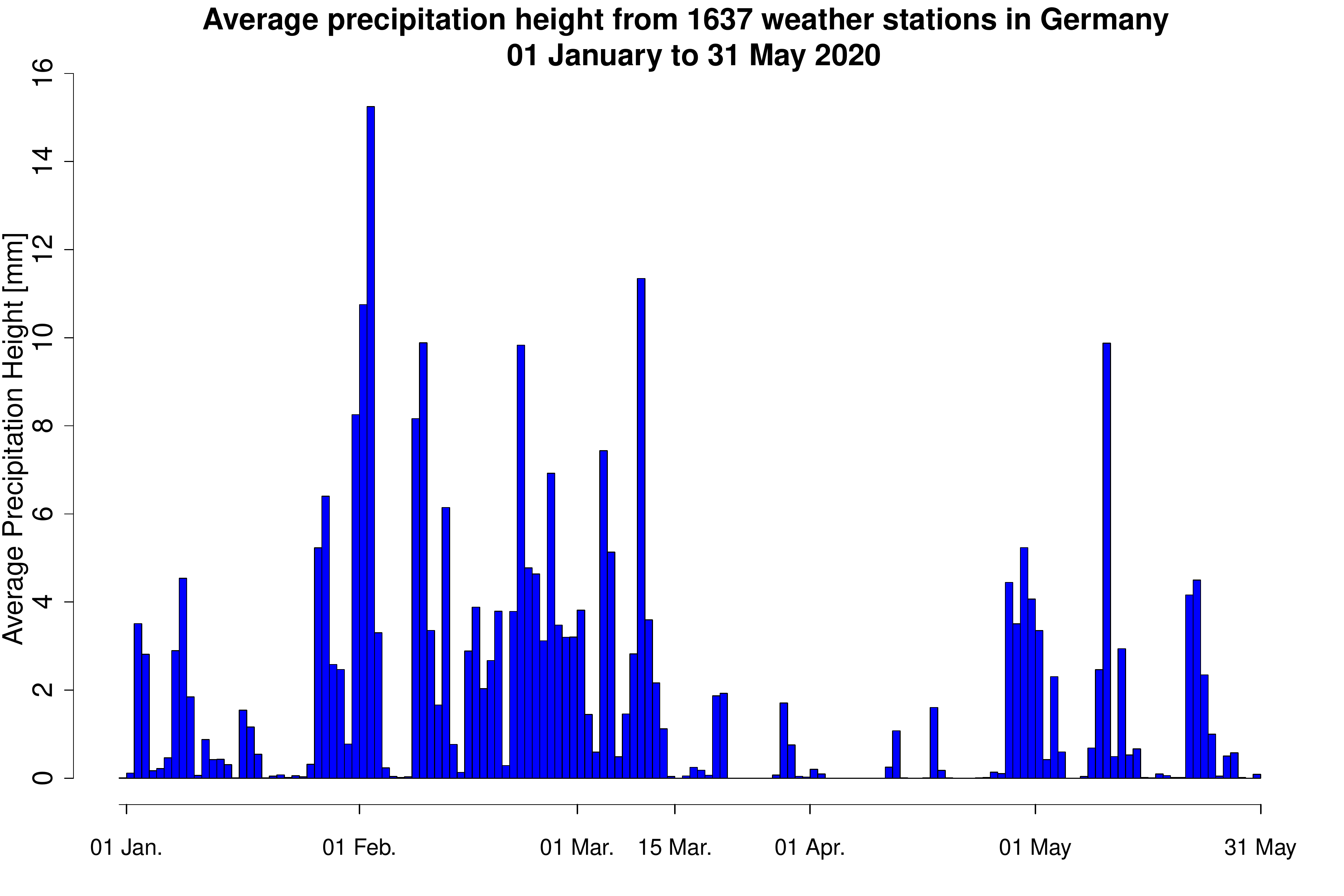} 
\caption{Daily rainfall (precipitation) in $mm$ in Germany averaged over $1637$ weather stations.}
\label{Niederschlag}
\end{figure}
\end{center}

Comparing this findings with Figure \ref{Feinstaub}, we can see that it fits the data quite well. Especially in February and the first half of March, with higher quantity of precipitation, we have relatively low quantity of $PM_{10}$. Beginning with the drought weather, the concentration of $PM_{10}$ goes up and especially the bottom-peaks are now higher than before, meaning that days with a concentration of $PM_{10}$ as low as in the beginning of the year are clearly more rare. 

We like to note that this findings do not contradict the satellite data published by ESA (e.g. \url{https://www.esa.int/Applications/Observing_the_Earth/Copernicus/Sentinel-5P/Air_pollution_remains_low_as_Europeans_stay_at_home}) which shows a reduced air pollution over Europe in 2020 compared to 2019. While the satellites measure atmospheric pollution, the data of the 'Umweltbundesamt' is collected at stations at ground level. It is known that there is a difference between these two sorts of pollution.

\subsection*{Simulation Study}
In this section we report the results of our simulation study. We compare size and power performance of our test statistic with the well established CUSUM. To do so, we construct different data examples which are described below. Note that we can easily adapt the bootstrap and the adapted bandwidth procedure described above to CUSUM by using $h(x,y)=x-y$ instead of the spatial sign kernel function $h(x,y)=(x-y)/\|x-y\|$.

\subsection*{Generating Sample}
We use a functional AR(1)-process on $[0,1]$, where the innovations are standard Brownian motions. We use an approximation on a finite grid with $d$ grid points, if not indicated otherwise. To be more precise, we simulate data as follows: 
\begin{align*}
& X_{-BI}=(\xi_1,\xi_1+\xi_2,...,\sum_{i=1}^d \xi_i)/\sqrt{d}, \;\;\;\xi_i \text{ i.i.d. }\mathcal{N}(0,1)\text{-distributed}\\
& X_t = a\, \Phi X_{t-1}^{\text{T}} + W_t \;\;\; \forall \; {-BI} <t \leq n \\
& \text{where } \Phi \in \mathbb{R}^{d\times d} \text{ with entries } \Phi_{i,j}=\begin{cases} i/d^2 & i\leq j \\ j/d^2 & i>j \end{cases} = \min(i,j)/d^2 \\
& \text{and } W_t = (\xi^{(t)}_1,\xi^{(t)}_1+\xi^{(t)}_2,...,\sum_{i=1}^d \xi^{(t)}_i)/\sqrt{d}, \;\;\; \xi^{(t)}_i \text{ i.i.d. }\mathcal{N}(0,1)\text{-distributed}
\end{align*}
The scalar $a \in \mathbb{R}$ is an AR-parameter, we use $a=1$.
The first $(BI+1)$ simulations are not used. Through this simulation structure we achieve dependence within $n$ and $d$. We consider $n=200$ and $d=100$ if not shated otherwise.

\subsection*{Size}
To calculate the empirical size, data simulation and test procedure via bootstrap is repeated $S=3000$ times with $m=1000$ bootstrap repetitions. We count the number of times the null hypothesis was rejected both for the CUSUM-type and the Wilcoxon-type statistic. By using $S=3000$ simulation runs, the standard deviation of the rejection frequencies is always below 1\% and is below 0.4\% if the true rejection probability is at 5\%.

To analyse how good the test statistics performs if outliers are present or if gaussianity is not given, we study two additional simulations: 
\begin{itemize}
\item Data simulated as above, but with presence of outliers: \[ Y_i=\begin{cases} X_i \;\;\; & i \notin \{ 0.2n,0.4n, 0.6n,0.8n\} \\ 
10X_i & i \in \{ 0.2n,0.4n,0.6n,0.8n\} \end{cases} \]
\item Data simulated similar to the above, but with $\xi_i, \xi_i^{(t)} \sim t_1  \, \forall i\leq d,$ \\$-BI<t\leq n$, i.e. heavy tailed data. 
\end{itemize}

As we can see in Table \ref{T_N2}, the Wilcoxon-type test and the CUSUM test perform almost similarly under normality, both are somewhat undersized, especially for a smaller size of $n=100$, but also for $n=200$ or $n=250$. In the presence of outliers or for heavy-tailed data, the rejection frequency of the Wilcoxon-type test does not change much, see Table \ref{T_N}. In contrast, the CUSUM test is very conservative in these situations.
\begin{table}[h]
\centering
\resizebox{\columnwidth}{!}{%
\begin{tabular}{ |c|c|c|c|c|c|c|c|c| } 
\hline
\multicolumn{7}{|c|}{Empirical Size} \\
\hline
 & \multicolumn{2}{|c|}{Gaussian $n=100$} & \multicolumn{2}{|c|}{Gaussian $n=200$}  & \multicolumn{2}{|c|}{Gaussian $n=250$} \\
 \hline
$\alpha$ &   CUSUM   & Spatial Sign &    CUSUM  & Spatial Sign & CUSUM     & Spatial Sign\\
\hline
 $0.1$   & $0.052$   & $0.057 $ & $0.080$ & $0.078$  & $0.082$  & $0.079$ \\ 
 $0.05$  & $0.013$ & $0.013$ & $0.033$ & $0.032$   & $0.031$ &   $0.029$ \\
 $0.025$ & $0.003$ & $0.003$ &$0.008$   & $0.001$    & $0.010$ & $0.008$ 	\\
 $0.01$  & $0$& $0$ & $0.002$   & $0.002$   & $0.003$   & $0.002$ \\ 
\hline
\end{tabular}%
}
\vspace{3pt}
\caption{Empirical size of CUSUM and spatial sign test with Gaussian data, significance level $\alpha$ and different sample sizes $n$.}
\label{T_N2}
\end{table}

\begin{table}[h]
\centering
\resizebox{\columnwidth}{!}{%
\begin{tabular}{ |c|c|c|c|c|c|c|c|c| } 
\hline
\multicolumn{7}{|c|}{Empirical Size} \\
\hline
 & \multicolumn{2}{|c|}{Gaussian} & \multicolumn{2}{|c|}{outlier}  & \multicolumn{2}{|c|}{heavy tails} \\
 \hline
$\alpha$ &   CUSUM   & Spatial Sign &    CUSUM  & Spatial Sign & CUSUM     & Spatial Sign\\
\hline
 $0.1$   & $0.080$ & $0.078$    & $0.051$ & $0.086$     	 & $0.018$   & $0.077$    	  \\ 
 $0.05$  & $0.033$ & $0.032$    & $0.015$   & $0.035$      & $0.003$ & $0.030$   \\
 $0.025$ & $0.008$   & $0.001$    & $0.004$ & $0.012$       & $0$ & $0.010$  	\\
 $0.01$  & $0.002$   & $0.002$      & $0.001$       & $0.003$      & $0$       & $0.002$   	\\ 
\hline
\end{tabular}%
}
\vspace{3pt}
\caption{Empirical size of CUSUM and spatial sign test with significance level $\alpha$, sample size $n=200$ and different distributions.}
\label{T_N}
\end{table}

\subsection*{Power} 
To evaluate the performance of the test statistics in presence of a change in mean, we construct four scenarios. 
\setlength{\leftmargini}{2cm}
\begin{itemize}
\item[Scenario 1:] Uniform jump of $+0.3$ after $n/2$ observations:
\[ Y_i=\begin{cases} X_i \;\;\; & i<n/2 \\ X_i +0.3u & i\geq n/2  \end{cases} \]
where $u=(1,...,1)^t$.
\item[Scenario 2:] Sinus-jump after $n/2$ of observations:
\[ Y_i=\begin{cases} X_i \;\;\; & i<n/2 \\ X_i + \frac{1}{2\sqrt{2}} (\sin(\pi D/d))_{D\leq d} & i\geq n/2  \end{cases}   \]
\item[Scenario 3:] Uniform jump of $+0.3$ after $n/2$ observations in presence of outlier at $0.2n,0.4n,0.6n,0.8n$:
\[ Y_i=\begin{cases} X_i \;\;\; & i<n/2, i \notin \{ 0.2n,0.4n\} \\ 
10X_i & i \in \{ 0.2n,0.4n\} \\
X_i +0.3u & i\geq n/2, i \notin \{ 0.6n,0.8n\}  \\ 
10X_i +0.3u &  i \in \{ 0.6n,0.8n\} \end{cases}
 \]
 \item[Scenario 4:] Heavy tails - In the simulation of $(X_i)_{i \leq n}$ we use $\xi_i, \xi_i^{(t)} \sim t_1$ (Cauchy distributed) $\forall i\leq d,-BI<t\leq n$ and a uniform jump of $+5$ after $n/2$ observations 
\end{itemize}
As in the analysis under null hypothesis $H_0$, we chose $m=1000$ bootstrap repetitions. The data simulation and test procedure via bootstrap is repeated $S=3000$ times for each scenario and the number of times $H_0$ was rejected is counted to calculate the empirical power. To compare our test-statistic with CUSUM, we calculate the Wilcoxon-type test (spatial sign) and the CUSUM test simultaneously in each simulation run. 

Comparing the size-power plots for both test statistics (Figure \ref{S_P}), we see that the Wilcoxon-type test outperforms the CUSUM test in Scenarios 1 and 2. For these two scenarios with a jump after one half of the observations, Wicoxon-type test provides similar empirical size and at the same time higher empirical power. 
In the third scenario, the jump with outlier in the data, we see that the CUSUM test shows a lower empirical size than the Wilcoxon-type test. But the spatial sign based test shows clearly more empirical power. 
In Scenario 4, we see that the CUSUM test barely provides any empirical power at all. Even for $\alpha=0.1$ CUSUM shows an empirical power $<0.04$. In heavy contrast, the Wilcoxon-type test shows relatively large empirical power, being greater than $0.9$ for $\alpha\geq 0.025$. 

For exact values of the empirical power in each scenario, see Table \ref{T_A123} in the appendix. In the appendix can also be found a short examination of the behaviour of the test statistics if the change-point lies more closely to the beginning of the observations or if $d$ is larger than $n$ (Table \ref{T_A56}). Here shall just be noted that the spatial sign based test suffers less loss in power than the CUSUM test if the change point lies closer to the edges or if $d>>n$. 

\begin{figure}[H]
\begin{center}
\includegraphics[width=\textwidth]{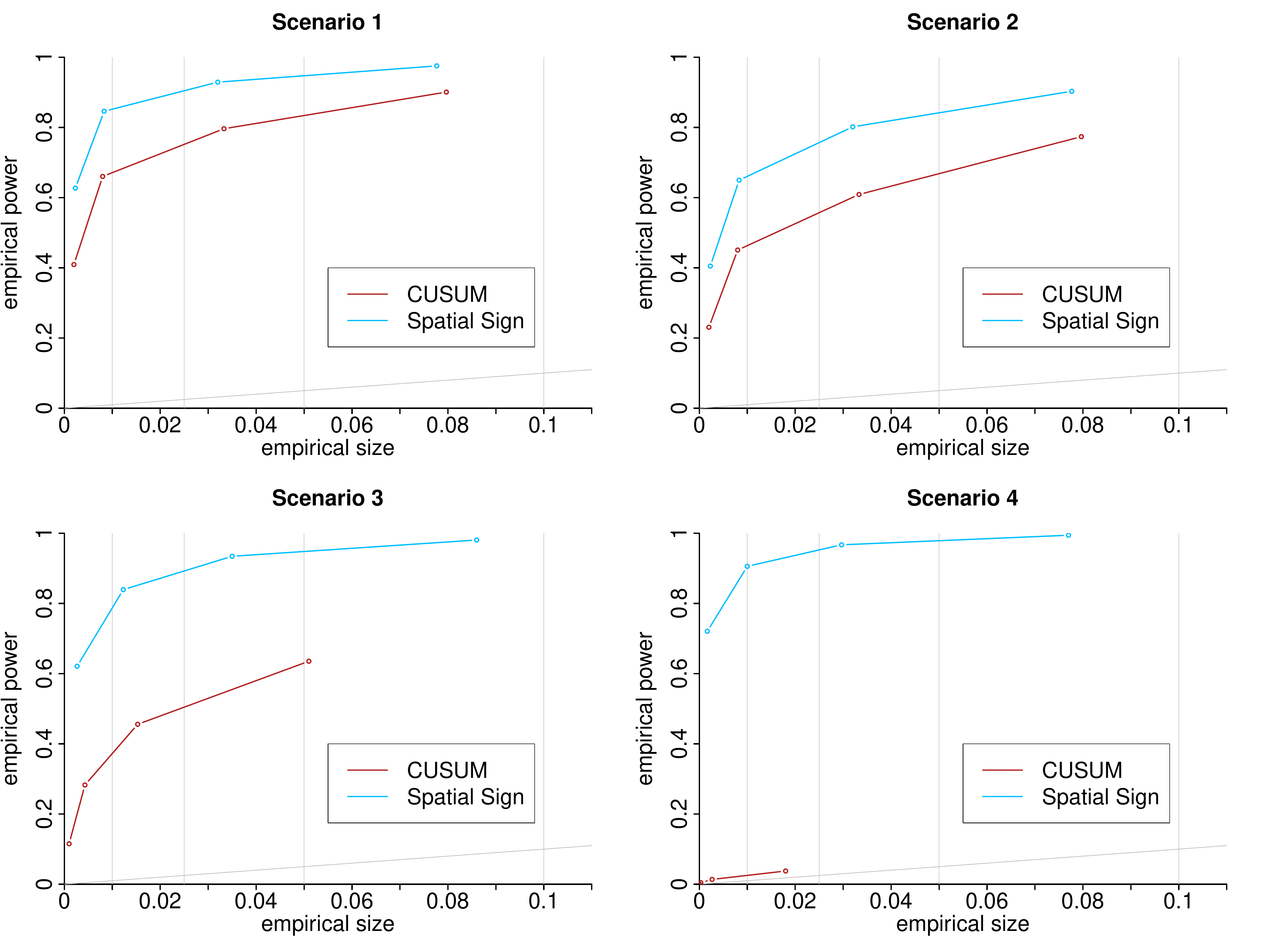}
\caption{Size-Power-Plot for CUSUM and Spatial Sign, Scenario 1-4.}
\label{S_P}
\end{center}
\end{figure}

\section{Auxilary Results}

\subsection{Hoeffding Decomposition and Linear Part}

The proofs will make use of Hoeffding's decomposition of the kernel $h$, so recall that Hoeffding's decomposition of $h$ is defined as
\[ h(x,y)= h_1(x)-h_1(y)+h_2(x,y) \, \forall x,y \in H,  \]
where \[h_1(x) = \E[h(x,\tilde{X})]  \]
\[h_2(x,y)= h(x,y) - \E[h(x,\tilde{X})] - \E[h(X,y)] = h(x,y) -h_1(x)+h_1(y)  \]
where $X,\tilde{X}$ are independent copies of $X_0$.  It is well known that $h_2$ is degenerate, that means $ \E[h_2(x,\tilde{X})]=\E[h_2(X,y)]=0$, see e.g. Section 1.6 in the book of \cite{lee2019u}.

\begin{lem}[Hoeffding's decomposition of $U_{n,k}$] \label{Hoeff}
Let $h:H\times H \rightarrow {H}$ be an antisymmetric kernel. Under Hoeffding's decomposition it holds for the test statistic that 
\[ U_{n,k}= \sum_{i=1}^k \sum_{j=k+1}^n h(X_i,X_j) = \underbrace{n \sum_{i=1}^k(h_1(X_i)-\overline{h_1(X)})}_{\text{linear part}} + \underbrace{\sum_{i=1}^k\sum_{j=k+1}^n h_2(X_i,X_j)}_{\text{degenerate part} }\]
where $\overline{h_1(X)} = \frac{1}{n}\sum_{j=1}^n h_1(X_j). $
\end{lem}

\begin{proof}
To prove the formula for $U_{n,k}$, we use Hoeffding's decomposition for $h$: 
\begin{align*}
 U_{n,k} &=\sum_{i=1}^k \sum_{j=k+1}^n h(X_i,X_j) = \sum_{i=1}^k\sum_{j=k+1}^n [h_1(X_i)-h_1(X_j)+h_2(X_i,X_j) ]  \\
&= \sum_{i=1}^k \sum_{j=k+1}^n [h_1(X_i)-h_1(X_j)] + \sum_{i=1}^k \sum_{j=k+1}^n h_2(X_i,X_j) \\
&= (n-k)h_1(X_1)-\sum_{j=k+1}^n h_1(X_j)+...+(n-k)h_1(X_k)-\sum_{j=k+1}^n h_1(X_j) \\ & \;\;\;\;\;\; + \sum_{i=1}^k \sum_{j=k+1}^n h_2(X_i,X_j)  \\
&= n h_1(X_1)-\sum_{j=1}^n h_1(X_j)+...+n h_1(X_k)-\sum_{j=1}^n h_1(X_j) \\
& \;\;\;\;\;\; + \sum_{i=1}^k \sum_{j=k+1}^n h_2(X_i,X_j) \\
&= n\Big(\sum_{i=1}^k[h_1(X_i)-\frac{1}{n}\sum_{j=1}^n h_1(X_j)]\Big)+\sum_{i=1}^k \sum_{j=k+1}^n h_2(X_i,X_j) \\
&= n \sum_{i=1}^k \Big(h_1(X_i)-\overline{h_1(X)}\Big) + \sum_{i=1}^k \sum_{j=k+1}^n h_2(X_i,X_j).
\end{align*}
\end{proof}
To use existing results about partial sums, we need to investigate the properties of the sequence $(h_1(X_n))_{n\in\Z}$.
\begin{lem}\label{linLem}
Under the assumptions of Theorem \ref{theo1}, $(h_1(X_n))_{n\in\mathbb{Z}}$ is $L_2$-NED with approximation constants $a_{k,2}=\mathcal{O}(k^{-4\frac{\delta+3}{\delta}})$.
\end{lem}
\begin{proof}
By Hoeffding's decomposition for $h$ it holds that $\forall x,x' \in H$
\begin{equation*}
\Vert h_1(x)-h_1(x')\Vert =\Vert \mathbb{E}[h(x,\tilde{X})]-\E[h(x',\tilde{X})]\Vert 
\end{equation*} 
Let $X,\tilde{X}$ be independent copies of $X_0$. Then by Jensen's inequality for conditional expectations and the variation condition
\begin{align}\label{2}
& \E\bigg[\Big( \sup_{\Vert x-X\Vert\leq\epsilon} \Vert h_1(x)-h_1(X)\Vert_{{H}}\Big)^2\bigg] \\
 = & \E\bigg[\Big(\sup_{\Vert x-X\Vert\leq\epsilon} E\big[\Vert h(x,\tilde{X})-h(X,\tilde{X})\Vert\big| X\big]\Big)^2\bigg] \nonumber\\
 \leq & \E\bigg[\Big(\sup_{\Vert x-X\Vert\leq\epsilon} \Vert h(x,\tilde{X})-h(X,\tilde{X})\Vert\Big)^2\bigg] \nonumber \\
 \leq &\E\bigg[\Big(\sup_{\substack{\Vert x-X\Vert\leq\epsilon \\ \Vert y-\tilde{X}\Vert\leq\epsilon }} \Vert h(x,y)-h(X,\tilde{X})\Vert \Big)^2  \bigg] \leq L\epsilon. \nonumber
\end{align}
We introduce the following notation: Let $X_{n,k}=f_k(\zeta_{n-k},...,\zeta_{n+k})$ and $\tilde{X}_{n,k}$ and independent copy of this random variable. Now, we can find the approximation constants of $(h_1(X_n))_n$ by using \eqref{2} and some further inequalities:
\begin{align*}
& \E[\Vert h_1(X_0)-\E[h_1(X_0)|\mathfrak{F}_{-k}^k]\Vert^2] \leq \E[\Vert h_1(X_0)-h_1(X_{0,k})\Vert^2] \\
&= \E[ \Vert h_1(X_0)-h_1(X_{0,k})\Vert^2 \textbf{1}_{\{\Vert X_0-X_{0,k}\Vert >s_k  \}} ]\\
& \hspace{100pt}+\E[ \Vert h_1(X_0)-h_1(X_{0,k})\Vert ^2 \textbf{1}_{\{\Vert X_0-X_{0,k}\Vert \leq s_k  \}}] \\
& \leq \E[ \Vert h_1(X_0)-h_1(X_{0,k})\Vert^2 \textbf{1}_{\{\Vert X_0-X_{0,k}\Vert >s_k  \}} ] \\
& \hspace{100pt}+ \underbrace{\E\bigg[\Big(\sup_{\Vert X_0-X_{0,k}\Vert \leq s_k} \Vert h_1(X_0)-h_1(X_{0,k})\Vert \Big)^2\bigg]}_{\overset{(\ref{2})}{\leq} Ls_k} \\
& \leq \left\Vert \Vert h_1(X_0)-h_1(X_0,k) \Vert ^2 \right\Vert_{\frac{2+\delta}{2}}+\left\Vert \textbf{1}_{\{\Vert X_0-X_{0,k}\Vert >s_k  \}} \right\Vert_{\frac{2+\delta}{\delta}} + L s_k \\
& \hspace{50pt}\text{by Hölder's inequality} \\
& = \left\Vert \Vert h_1(X_0)-h_1(X_0,k) \Vert ^2 \right\Vert_{\frac{2+\delta}{2}}+  \mathbb{P}(\Vert X_0-X_{0,k} \Vert > s_k)^{\frac{\delta}{2+\delta}} + Ls_k \displaybreak[0]\\
& \leq \E[ \Vert h_1(X_0)-h_1(X_{0,k})\Vert^{2+\delta} ]^{\frac{2}{2+\delta}} + (a_k\Phi(s_k))^{\frac{\delta}{2+\delta}} + L s_k \;\;\; \text{since $(X_n)_n$ is P-NED} \\
&=\E\left[\left\Vert \E[h(X_0,\tilde{X}_0) | X_0,X_{0,k}] - \E[h(X_{0,k},\tilde{X}_{0,k})|X_0,X_{0,k}] \right\Vert^{2+\delta}  \right]^\frac{2}{2+\delta} (a_k\Phi(s_k))^{\frac{\delta}{2+\delta}} \\
& \hspace{100pt} + L s_k \displaybreak[0]\\
& \leq \E\left[ \E[ \Vert h(X_0,\tilde{X}_0)-h(X_{0,k},\tilde{X}_{0,k}) \Vert^{2+\delta} | X_0,X_{0,k}] \right]^{\frac{2}{2+\delta}} (a_k\Phi(s_k))^{\frac{\delta}{2+\delta}} + Ls_k \\
& \hspace{50pt}\text{by Jensen's inequality} \\
&= \left( \E[\Vert h(X_0,\tilde{X}_0)-h(X_{0,k},\tilde{X}_{0,k}) \Vert ^{2+\delta} ]^{\frac{1}{2+\delta}} \right)^2 (a_k\Phi(s_k))^{\frac{\delta}{2+\delta}} +Ls_k \displaybreak[0]\\
& \leq \left( \E[\Vert h(X_0,\tilde{X}_0)\Vert^{2+\delta}]^{\frac{1}{2+\delta}} + \E[\Vert h(X_{0,k},\tilde{X}_{0,k})\Vert ^{2+\delta}]^{\frac{1}{2+\delta}} \right)^2 (a_k\Phi(s_k))^{\frac{\delta}{2+\delta}} +Ls_k  \\
& \hspace{50pt} \text{by Minkowski's inequality} \\
& \leq (M^{\frac{1}{2+\delta}}+M^{\frac{1}{2+\delta}} )^2 (a_k\Phi(s_k))^{\frac{\delta}{2+\delta}} +Ls_k \\
& \hspace{50pt} \text{by the uniform moment condition, choose $s_k=k^{-8\frac{3+\delta}{\delta}}$} \\
& \leq C (k^{-8 \frac{(3+\delta)(2+\delta)}{\delta^2}})^{\frac{\delta}{2+\delta}} +L k^{-8\frac{3+\delta}{\delta}} \;\;\; \text{by the assumption on the P-NED coefficients} \\
& = C k^{-8 \frac{3+\delta}{\delta}}.
\end{align*}
By taking the square root, we get the result:
\[ \left( \E[\Vert h_1(X_0)-\E[h_1(X_0)| \mathfrak{F}_{-k}^k]\Vert^2] \right)^{\frac{1}{2}} \leq C k^{-4 \frac{3+\delta}{\delta}} =: a_{k,2}. \]
Since it holds that $a_{k,2} \xrightarrow{k\to\infty} 0$, $(X_n)_{n\in\mathbb{Z}}$ is $L_2$-NED.
\end{proof}

\begin{prop} \label{linThm}
 Under Assumptions of Theorem \ref{theo1} it holds:
\[ \Big( \frac{1}{\sqrt{n}} \sum_{i=1}^{\left \lfloor{n\lambda}\right\rfloor  } h_1(X_i) \Big)_{\lambda\in [0,1]} \Rightarrow (W(\lambda))_{\lambda\in [0,1]}  \]
where $(W(\lambda))_{\lambda\in [0,1]}$ is  a Brownian motion with covariance operator as defined in Theorem \ref{theo1}.
\end{prop}
\begin{proof}
We want to use Theorem 1 \cite{STW16} for $(h_1(X_n))_{n\in\mathbb{Z}}$, so we have to check the assumptions: 

\underline{Assumption 1:} $(h_1(X_n))_{n\in\mathbb{Z}}$ is $L_1$-NED. \\
We know by Lemma \ref{linLem} that $(h_1(X_n))_{n\in\mathbb{Z}}$ is $L_2$-NED. Thus, $L_1$-NED follows by Jensen's inequality:
\begin{align*}
\E[\Vert h_1(X_0)-\E[h_1(X_0)|\mathfrak{F}_{-k}^k]\Vert] & \leq \E[\Vert h_1(X_0)-\E[h_1(X_0)| \mathfrak{F}_{-k}^k] \Vert^2]^{\frac{1}{2}}  \leq a_{k,2}
\end{align*}
So, $(h_1(X_n))_{n\in\mathbb{Z}}$ is $L_1$-NED with constants $a_{k,1}=a_{k,2}=Ck^{-4\frac{3+\delta}{\delta}}$.

\underline{Assumption 2:} Existing $(4+\delta)$-moments. \\
This follows from the assumption of uniform moments under approximation:
\begin{align*}
\E[\Vert h_1(X_0) \Vert ^{4+\delta}] & = \E[\Vert \E[h(X_0,\tilde{X}_0)| X_0] \Vert^{4+\delta} ] \\
& \leq \E[ \E[\Vert h(X_0,\tilde{X}_0) \Vert^{4+\delta} | X_0  ]]  \;\;\; \text{by Jensen's inequality} \\
& = \E[\Vert h(X_0,\tilde{X}_0) \Vert^{4+\delta}] \leq M < \infty
\end{align*}
In the case that $h$ is bounded, the same holds for $h_1$.

\underline{Assumption 3:} $\sum_{m=1}^{\infty} m^2 a_{m,1}^{\frac{\delta}{3+\delta}} < \infty$ \\
\begin{align*}
\sum_{m=1}^{\infty} m^2 a_{m,1}^{\frac{\delta}{3+\delta}} = C \sum_{m=1}^{\infty}m^2(m^{-4\frac{3+\delta}{\delta}})^{\frac{\delta}{3+\delta}} = C \sum_{m=1}^{\infty} m^2 m^{-4} = C \sum_{m=1}^{\infty} m^{-2} < \infty
\end{align*}

\underline{Assumption 4:} $\sum_{m=1}^{\infty} m^2 \beta_m^{\frac{\delta}{4+\delta}} < \infty$. \\
This holds directly by the assumed rate on the coefficients $\beta_m$.

We have checked that all assumptions for Theorem 1 \cite{STW16} are fulfilled and since $\E[h_1(X_0)]=0$ because $h$ is antisymmetric, the statement of the theorem follows. 
\end{proof}

\subsection{Degenerate part}

\begin{lem} \label{P1}
Under the assumptions of Theorem \ref{theo1}, there exists a universal constant $C>0$ such that for every $i,k,l\in\mathbb{N}$, $\epsilon>0$ it holds that
\[ \E[\Vert h_2(X_i,X_{i+k+2l})-h_2(X_{i,l},X_{i+k+2l,l})\Vert^2]^{\frac{1}{2}} \leq C(\sqrt{\epsilon}+\beta_k^{\frac{\delta}{2(2+\delta)}}+(a_l\Phi(\epsilon))^{\frac{\delta}{2(2+\delta)}} ),  \]
where $X_{i,l} = f_l(\zeta_{i-l},...,\zeta_{i+l})$.
\end{lem}
\begin{proof}
By Lemma D1 \cite{DVWW17} there exist copies $(\zeta'_n)_{n\in\mathbb{Z}}$, $(\zeta''_n)_{n\in\mathbb{Z}}$ of $(\zeta_n)_{n\in\mathbb{Z}}$ which are independent of each other and satisfy
\begin{align} \label{Z}
\Prob((\zeta'_n)_{n\geq i+k+l}=(\zeta_n)_{n\geq i+k+l})=1-\beta_k \;\;\; \text{and} \;\;\; \Prob((\zeta''_n)_{n\leq i+l}=(\zeta_n)_{n\leq i+l})=1-\beta_k
\end{align}
Define \begin{align*}
& X'_i= f( (\zeta'_{i+n})_{n\in\mathbb{Z}}) \; , \;\; \;  X''_i= f( (\zeta''_{i+n})_{n\in\mathbb{Z}}) \\
& X'_{i,l}=f_l(\zeta'_{i-l},...,\zeta'_{i+l}) \; , \;\;\; X''_{i,l}=f_l(\zeta''_{i-l},...,\zeta''_{i+l}) .
\end{align*}
With the help of these, we can write
\begin{align}
& \E[\Vert h_2(X_i,X_{i+k+2l})-h_2(X_{i,l},X_{i+k+2l,l}) \Vert^2]^{\frac{1}{2}} \nonumber\\
& \leq \E[\Vert h_2(X_i,X_{i+k+2l})-h_2(X''_i, X'_{i+k+2l}) \Vert ^2]^{\frac{1}{2}} \label{I}\\
& \;\;\; + \E[\Vert h_2(X''_i, X'_{i+k+2l})-h_2(X''_{i,l}, X'_{i+k+2l,l}) \Vert^2]^{\frac{1}{2}} \label{II}\\
& \;\;\; + \E[\Vert h_2(X''_{i,l}, X'_{i+k+2l,l})-h_2(X_{i,l},X_{i+k+2l,l})\Vert^2]^{\frac{1}{2}} \label{III}
\end{align}
by using the triangle inequality. We will look at the three summands separately. For abbreviation, we define 
\[B= \{     (\zeta'_n)_{n\geq i+k+l} = (\zeta_n)_{n\geq i+k+l}, \,(\zeta''_n)_{n\leq i+l} = (\zeta_n)_{n\leq i+l} \} \]
\[ B^c= \{     (\zeta'_n)_{n\geq i+k+l} \neq (\zeta_n)_{n\geq i+k+l} \text{ or } (\zeta''_n)_{n\leq i+l} \neq (\zeta_n)_{n\leq i+l} \}  \]
\begin{align*}
 (\ref{I})& = \E[\Vert h_2(X_i,X_{i+k+2l})-h_2(X''_i, X'_{i+k+2l}) \Vert ^2]^{\frac{1}{2}} \\
& \leq \;\; \E[\Vert h_2(X_i,X_{i+k+2l})-h_2(X''_i, X'_{i+k+2l}) \Vert ^2  \textbf{1}_{B^c}]^{\frac{1}{2}} \tag{\ref{I}.A} \label{IA}\\
& \;\;\; + \E[\Vert h_2(X_i,X_{i+k+2l})-h_2(X''_i, X'_{i+k+2l}) \Vert ^2   
\textbf{1}_{B}]^{\frac{1}{2}} \tag{\ref{I}.B} .\label{IB}
\end{align*}
 For (\ref{IA}), we use Hölder's inequality together with our assumptions on uniform moments under approximation and get
\begin{align*}
(\ref{IA}) & \leq \E[\Vert h_2(X_i,X_{i+k+2l})-h_2(X''_i,X'_{i+k+2l})\Vert^{\frac{2(2+\delta)}{2}}]^{\frac{2}{2(2+\delta)}}\Prob(B^c )^{\frac{\delta}{2(2+\delta)}} \\ 
& \leq \left( \E[\Vert h_2(X_i,X_{i+k+2l})\Vert^{2+\delta}]^{\frac{1}{2+\delta}} + \E[\Vert h_2(X''_i, X'_{i+k+2l})\Vert^{2+\delta}]^{\frac{1}{2+\delta}}\right) \\
& \hspace{20pt} \cdot \big( \Prob( \{ \zeta'_n)_{n\geq i+k+l} \neq (\zeta_n)_{n\geq i+k+l} \}) + \Prob( \{ (\zeta''_n)_{n\leq i+l} \neq (\zeta_n)_{n\leq i+l} \}) \big) ^{\frac{\delta}{2(2+\delta)}} \\
&\leq 2 M^\frac{1}{2+\delta}(2\beta_k^{\frac{\delta}{2(2+\delta)}}) \\
& \leq C \beta_k^{\frac{\delta}{2(2+\delta)}},
\end{align*}
where we used property \eqref{Z} of the copied series $(\zeta'_n)_{n\in\mathbb{Z}}$, $(\zeta''_n)_{n\in\mathbb{Z}}$ for the second to last inequality. For (\ref{IB}), we split up again:
\begin{align*}
(\ref{IB}) & \leq \; \E[\Vert h_2(X_i,X_{i+k+2l})-h_2(X''_i,X'_{i+k+2l})\Vert^2 \textbf{1}_B \\
& \hspace{100pt} \textbf{1}_{\{ \Vert X_i-X''_i\Vert \leq 2\epsilon, \, \Vert X_{i+k+2l}-X'_{i+k+2l} \Vert \leq 2\epsilon \}}  ]^{\frac{1}{2}} \\
& \;\; +\E[\Vert h_2(X_i,X_{i+k+2l})-h_2(X''_i,X'_{i+k+2l})\Vert^2 \textbf{1}_B \\
& \hspace{100pt} \textbf{1}_{\{ \Vert X_i-X''_i\Vert > 2\epsilon\, \text{or} \, \Vert X_{i+k+2l}-X'_{i+k+2l} \Vert > 2\epsilon \}}  ]^{\frac{1}{2}}.
\end{align*}
For the first summand, we use variation condition. For the second, notice that on $B$:
\[\Vert X_i-X''_i \Vert \leq \Vert X_i- X_{i,l}\Vert + \Vert X_{i,l}-X''_i \Vert =\Vert X_i -X_{i,l} \Vert +\Vert X''_{i,l}-X''_i \Vert \]
and 
\begin{align*}
\Vert X_{i+k+2l}-X'_{i+k+2l} \Vert & \leq \Vert X_{i+k+2l}- X_{i+k+2l,l}\Vert + \Vert X_{i+k+2l,l}-X'_{i+k+2l} \Vert \\
& =\Vert X_{i+k+2l} -X_{i+k+2l,l} \Vert +\Vert X'_{i+k+2l,l}-X'_{i+k+2l} \Vert.
\end{align*}
So,
\begin{align*}
(\ref{IB}) & \leq \sqrt{L2\epsilon} \\
&  \hspace{10pt} + \E[\Vert h_2(X_i,X_{i+k+2l})-h_2(X''_i,X'_{i+k+2l})\Vert^2 \textbf{1}_{\{\Vert X_i-X_{i,l} \Vert > \epsilon  \}}]^{\frac{1}{2}} \\
& \hspace{10pt} +  \E[\Vert h_2(X_i,X_{i+k+2l})-h_2(X''_i,X'_{i+k+2l})\Vert^2 \textbf{1}_{\{\Vert X''_i-X''_{i,l} \Vert > \epsilon  \}}]^{\frac{1}{2}} \\
& \hspace{10pt} + \E[\Vert h_2(X_i,X_{i+k+2l})-h_2(X''_i,X'_{i+k+2l})\Vert^2 \textbf{1}_{\{\Vert X_{i+k+2l}-X_{i+k+2l,l} \Vert > \epsilon  \}}]^{\frac{1}{2}} \\
& \hspace{10pt} + \E[\Vert h_2(X_i,X_{i+k+2l})-h_2(X''_i,X'_{i+k+2l})\Vert^2 \textbf{1}_{\{\Vert X'_{i+k+2l}-X'_{i+k+2l,l} \Vert > \epsilon  \}}]^{\frac{1}{2}} \\
& \leq \sqrt{L2\epsilon}+ 4 \cdot 2M^{\frac{1}{2+\delta}}(\Prob(\Vert X_i-X_{i,l}\Vert > \epsilon))^{\frac{\delta}{2(2+\delta)}} \\
& \hspace{50pt}\text{by our moment assumptions and Hölder's inequality} \\
& \leq \sqrt{L2\epsilon} +4 \cdot 2M^{\frac{1}{2+\delta}} (a_l\Phi(\epsilon))^{\frac{\delta}{2(2+\delta)}} \;\;\; \text{since $(X_n)_{n\in\Z}$ is P-NED} \\
& \leq C\left( \sqrt{\epsilon}+(a_l\Phi(\epsilon))^{\frac{\delta}{2(2+\delta)}} \right)
\end{align*}
Combining the results for (\ref{IA}) and (\ref{IB}) we get 
\[(\ref{I}) \leq (\ref{IA})+(\ref{IB}) \leq C \left( \beta_k^{\frac{\delta}{2(2+\delta)}} +\sqrt{\epsilon}+(a_l\Phi(\epsilon))^{\frac{\delta}{2(2+\delta)}} \right).\]

We can now look at (\ref{II}). Again, we split the term into two summands, (similar as for (\ref{I})) we use the variation condition for the first and Hölder's inequality for the second summand:
\begin{align*}
(\ref{II}) & = \E[\Vert h_2(X''_i, X'_{i+k+2l})-h_2(X''_{i,l}, X'_{i+k+2l,l}) \Vert^2  ]^{\frac{1}{2}} \\
& \leq \;\;\;\; \E[\Vert h_2(X''_i, X'_{i+k+2l})-h_2(X''_{i,l}, X'_{i+k+2l,l}) \Vert^2 \\
& \hspace{100pt}\textbf{1}_{\{ \Vert X''_i-X''_{i,l}\Vert \leq \epsilon,\, \Vert X'_{i+k+2l}-X'_{i+k+2l,l} \Vert \leq\epsilon\}} ]^{\frac{1}{2}} \\
& \hspace{11pt} + \E[\Vert h_2(X''_i, X'_{i+k+2l})-h_2(X''_{i,l}, X'_{i+k+2l,l}) \Vert^2 \\
& \hspace{100pt} \textbf{1}_{\{ \Vert X''_i-X''_{i,l}\Vert > \epsilon\, \text{or} \, \Vert X'_{i+k+2l}-X'_{i+k+2l,l} \Vert > \epsilon\}} ]^{\frac{1}{2}} \\
& \leq \sqrt{L\epsilon} + \left( \E[\Vert h_2(X''_i,X'_{i+k+2l}) \Vert ^{2+\delta}]^{\frac{1}{2+\delta}} +\E[\Vert h_2(X''_{i,l},X'_{i+k+2l,l})\Vert^{2+\delta}]^{\frac{1}{2+\delta}}\right) \\
& \hspace{50pt} \cdot \left(\Prob(\Vert X''_i-X''_{i,l}\Vert > \epsilon) + \Prob(\Vert X'_{i+k+2l}-X'_{i+k+2l,l} \Vert > \epsilon) \right)^{\frac{\delta}{2(2+\delta)}} \\
& \leq \sqrt{L\epsilon} + 2M^{\frac{1}{2+\delta}}(2a_l\Phi(\epsilon))^{\frac{\delta}{2(2+\delta)}} \;\; \text{since $(X_n)_{n \in \mathbb{Z}}$ is P-NED} \\
& \leq C\left(\sqrt{\epsilon} + (a_l\Phi(\epsilon))^{\frac{\delta}{2(2+\delta)}}\right)
\end{align*}
Lastly, we split up (\ref{III}) as well:
\begin{align*}
(\ref{III}) & = \E[ \Vert h_2(X''_{i,l},X'_{i+k+2l,l})-h_2(X_{i,l},X_{i+k+2l,l})\Vert^2]^{\frac{1}{2}}  \\
& \leq \;\;\; \E[ \Vert h_2(X''_{i,l},X'_{i+k+2l,l})-h_2(X_{i,l},X_{i+k+2l,l})\Vert^2 \textbf{1}_{B^c}]^{\frac{1}{2}} \\
& \hspace{11pt} + \E[ \Vert h_2(X''_{i,l},X'_{i+k+2l,l})-h_2(X_{i,l},X_{i+k+2l,l})\Vert^2 \textbf{1}_{B}]^{\frac{1}{2}}.
\end{align*}
Since on $B$ it is $X_{i+k+2l,l}=X'_{i+k+2l,l}$ and $X_{i,l}=X''_{i,l}$, the second summand equals zero. For the first summand, we use Hölder's inequality again and the properties of $(\zeta'_n)_{n\leq i+l}$, $(\zeta''_n)_{n\leq i+l}$, see (\ref{Z}): 
\begin{align*}
(\ref{III}) & \leq 2M^{\frac{1}{2+\delta}} \big(\Prob(\{(\zeta'_n)_{n \geq i+k+l} \neq (\zeta_n)_{n \geq i+k+l}  \}) \!+\! \Prob( \{(\zeta''_n)_{n\leq i+l} \neq (\zeta_n)_{n\leq i+l}  \})  \big)^{\frac{\delta}{2(2+\delta)}} \\
&\leq 2M^{\frac{1}{2+\delta}}(2\beta_k)^{\frac{\delta}{2(2+\delta)}}  \leq C \beta_k^{\frac{\delta}{2(2+\delta)}} 
\end{align*}
We can finally put everything together: 
\begin{align*}
& \E[\Vert h_2(X_i,X_{i+k+2l})-h_2(X_{i,l}, X_{i+k+2l,l})\Vert^2]^{\frac{1}{2}} \leq (\ref{I})+(\ref{II})+(\ref{III}) \\
& \leq C \left(\beta_k^{\frac{\delta}{2(2+\delta)}}+\sqrt{\epsilon}+(a_l\Phi(\epsilon))^{\frac{\delta}{2(2+\delta)}}\right) + C\left(\sqrt{\epsilon}+ (a_l\Phi(\epsilon))^{\frac{\delta}{2(2+\delta)}} \right) + C \beta_k^{\frac{\delta}{2(2+\delta)}} \\
& \leq C \left( \sqrt{\epsilon} + \beta_k^{\frac{\delta}{2(2+\delta)}} + (a_l\Phi(\epsilon))^{\frac{\delta}{2(2+\delta)}}\right)
\end{align*}
\end{proof}

\begin{lem} \label{P2}
Under the assumptions of Theorem \ref{theo1} it holds for any $n_1 < n_2 < n_3 < n_4$ and $l= \left \lfloor{n_4^{\frac{3}{16}}}\right\rfloor $:
\[\E\bigg[\Big(\sum_{n_1 \leq i \leq n_2}\sum_{n_3 \leq j \leq n_4} \Vert h_2(X_i,X_j)-h_2(X_{i,l},X_{j,l})\Vert \Big)^2 \bigg]\wur \leq C(n_4-n_3)n_4^{\frac{1}{4}} \] 
\end{lem}
\begin{proof} The important step of the proof is to bound the left hand side expectation from above by a sum of $\E[\| h_2(X_i,X_j) - h_2(X_{i,l},Y_{j,l} )\|^2]^{1/2}$ terms. We can then use Lemma \ref{P1} to achieve the stated approximation.
First note that 
\begin{align*}
 & E\bigg[\Big(\sum_{n_1 \leq i \leq n_2}\sum_{n_3 \leq j \leq n_4} \Vert h_2(X_i,X_j)-h_2(X_{i,l},X_{j,l})\Vert \Big)^2 \bigg]\wur \\
 & \leq E\bigg[\Big(\sum_{1 \leq i \leq j-1}\sum_{n_3 \leq j \leq n_4} \Vert h_2(X_i,X_j)-h_2(X_{i,l},X_{j,l})\Vert \Big)^2 \bigg]\wur.
\end{align*}
For any fixed $j$ it is
\[ \E\bigg[ \sum_{1\leq i < j} \Vert h_2(X_i,X_j) \Vert \bigg] = \E\bigg[ \sum_{k=1}^{j-1} \Vert h_2(X_{j-k},X_j) \Vert \bigg] \leq \E\bigg[ \sum_{k=1}^{n_4} \Vert h_2(X_{j-k},X_j) \Vert \bigg]  .\]
And for $j$ there are at most $(n_4-n_3)$ possibilities. So
\[ \E \bigg[ \sum_{n_3 \leq j \leq n_4} \sum_{1 \leq i < j} \Vert h_2(X_i,X_j) \Vert \bigg] \leq (n_4-n_3)  \E \bigg[ \sum_{k=1}^{n_4} \Vert h_2(X_{j-k},X_j) \Vert \bigg]  .\]
The analog holds for $h_2(X_{i,l},X_{j,l})$. Thus,
\begin{align}
& \E\bigg[\Big(\sum_{1 \leq i < j, n_3\leq j \leq n_4} \Vert h_2(X_i,X_j)-h_2(X_{i,l},X_{j,l})\Vert \Big)^2 \bigg]\wur \nonumber\\
& \leq \sum_{n_3\leq j\leq n_4} \sum_{1\leq i<j} \E[ \Vert h_2(X_i,X_j)-h_2(X_{i,l},X_{j,l})\Vert \nonumber ^2]\wur \nonumber\\
& \leq (n_4-n_3) \sum_{k=1}^{n_4} \E[\Vert h_2(X_{j-k},X_j)-h_2(X_{j-k,l},X_{j,l}) \Vert^2  ]\wur \nonumber \\
& \leq (n_4-n_3) \sum_{k=1}^{n_4} C \left( \sqrt{\epsilon} +\beta_{k-2l}^{\frac{\delta}{2(2+\delta)}} + (a_l\Phi(\epsilon))\del  \right) \;\; \text{by Lemma \ref{P1}.} \label{3}
\end{align}
Now set $\epsilon=l^{-8\frac{3+\delta}{\delta}}$ and define $\beta_k=1$ if $k<0$. Then by our assumptions on the approximation constants and the mixing coefficients 
\begin{align*}
(\ref{3})& = C(n_4-n_3)\sum_{k=1}^{n_4} \left( l^{-8\frac{3+\delta}{\delta}\frac{1}{2}}+\beta_{k-2l}\del +(a_l\Phi(l^{-8\frac{3+\delta}{\delta}}))\del \right) \\
& \leq C(n_4-n_3) \sum_{k=1}^{n_4} \left( l^{-4\frac{3+\delta}{\delta}}+\beta_{k-2l}\del +l^{-4\frac{3+\delta}{\delta}} \right) \\
& \leq C (n_4-n_3) \Big( \sum_{k=1}^{n_4} l^{-4} + \sum_{k=1}^{2l-1} \underbrace{\beta_{k-2l}^{\frac{\delta}{4+\delta}}}_{=1} +\sum_{k=2l}^{n_4} \beta_{k-2l}^{\frac{\delta}{4+\delta}} \Big) \\
& \leq C (n_4-n_3) \Big(n_4 l^{-4}+2l + \underbrace{\sum_{k=2l}^{n_4} (k-2l)^2 \beta_{k-2l}^{\frac{\delta}{4+\delta}}}_{< \infty } \Big)\\
&\leq C (n_4-n_3) n_4^{\frac{1}{4}}.
\end{align*}
So the statement of the lemma is proven.
\end{proof}

\begin{lem}\label{P3}
Under the assumptions of Theorem \ref{theo1}, it holds for any $n_1<n_2<n_3<n_4$ and $l= \left \lfloor{n_4^{\frac{3}{16}}}\right\rfloor $: 
\[ \E\bigg[\Big( \sum_{n_1\leq i \leq n_2,\, n_3\leq j\leq n_4} \Vert h_{2,l}(X_{i,l},X_{j,l})-h_2(X_{i,l},X_{j,l}) \Vert \Big)^2\bigg]\wur \leq C (n_4-n_3) n_4^{\frac{1}{4}}  \]
where $h_{2,l}(x,y)= h(x,y)-\E[h(x,\tilde{X}_{j,l})]-\E[h(\tilde{X}_{i,l},y)]\;\;\; \forall i,j,\in \mathbb{N}$ and $\tilde{X}_{i,l}=f_l(\tilde{\zeta}_{i-l},...,\tilde{\zeta}_{i+l})$, where $(\tilde{\zeta}_n)_{n\in\mathbb{\zeta}}$ is an independent copy of  $(\zeta_n)_{n\in\mathbb{\zeta}}$.
\end{lem}
\begin{proof}
For $(\tilde{\zeta}_n)_{n\in\mathbb{Z}}$ an independent copy of $(\zeta_n)_{n\in\mathbb{\zeta}}$, write $\tilde{X}_i=f((\tilde{\zeta}_{i+n})_{n\in\mathbb{Z}})$. So $(\tilde{X}_i)_{i\in\mathbb{Z}}$ is an independent copy of $\Xn$. We will use Hoeffding's decomposition and rewrite $h_2$ as $h_2(x,y)=h(x,y)-\E[h(x,\tilde{X}_j)]-\E[h(\tilde{X}_i,y)]$ and similarly for $h_{2,l}$. By doing so, we obtain
\begin{align}
& \E[\Vert h_{2,l}(X_{i,l},X_{j,l})-h_2(X_{i,l},X_{j,l}) \Vert^2]\wur \nonumber\\
& = \E[\Vert \;\; h(X_{i,l},X_{j,l})- \E_{\tilde{X}}[h(X_{i,l},\tilde{X}_{j,l})]- \E_{\tilde{X}}[h(\tilde{X}_{i,l},X_{j,l})] \nonumber\\
& \hspace{20pt} - h(X_{i,l},X_{j,l})+ \E_{\tilde{X}}[h(X_{i,l},\tilde{X}_j)]+\E_{\tilde{X}}[h(\tilde{X}_i,X_{j,l})] \Vert^2 ]\wur \nonumber \\
& \leq \;\; \E[\Vert h(X_{i,l},\tilde{X}_{j,l})-h(X_{i,l},\tilde{X}_j) \Vert^2 ]\wur \label{2I} \\
&  \hspace{11pt} +\E[\Vert h(\tilde{X}_{i,l},X_{j,l})-h(\tilde{X}_i,X_{j,l}) \Vert^2]\wur .\label{2II}
\end{align}
Here $\E_{\tilde{X}}$ denotes the expectation with respect to $\tilde{X}$, $\E=\E_{X,\tilde{X}}$ is the expectation with respect to $X$ and $\tilde{X}$.
We bound the two terms separately, starting with $(\ref{2II})$:
\begin{align*}
&\E[\Vert h(\tilde{X}_{i,l},X_{j,l})-h(\tilde{X}_i,X_{j,l}) \Vert^2]\wur \\
& \leq \E[\Vert h(\tilde{X}_{i,l},X_{j,l})-h(\tilde{X}_i,X_j) \Vert^2  ]\wur \tag{\ref{2II}.A}  \label{2II.A} \\
& + \E[\Vert h(\tilde{X}_i,X_{j,l})-h(\tilde{X}_i,X_j)\Vert^2]\wur \tag{\ref{2II}.B} \label{2II.B}
\end{align*}
Now, for the first summand, we obtain
\begin{align*}
(\ref{2II.A})& = \E[\Vert h(\tilde{X}_{i,l},X_{j,l})-h(\tilde{X}_i,X_j) \Vert^2  \textbf{1}_{\{\Vert \tilde{X}_i-\tilde{X}_{i,l}\Vert\leq\epsilon,\; \Vert X_j-X_{j,l}\Vert \leq \epsilon   \}}]\wur \\
& + \E[\Vert h(\tilde{X}_{i,l},X_{j,l})-h(\tilde{X}_i,X_j) \Vert^2  \textbf{1}_{\{\Vert \tilde{X}_i-\tilde{X}_{i,l}\Vert > \epsilon \, \text{or} \, \Vert X_j-X_{j,l}\Vert > \epsilon   \}}]\wur \\
& \leq \sqrt{L\epsilon}+\E[\Vert h(\tilde{X}_{i,l},X_{j,l})-h(\tilde{X}_i,X_j)\Vert^{2+\delta}  ]^{\frac{1}{2+\delta}} \\
& \hspace{70pt}\cdot \left(\Prob(\Vert \tilde{X}_i-\tilde{X}_{i,l} \Vert > \epsilon )+\Prob(\Vert X_j-X_{j,l} \Vert > \epsilon)\right)
\end{align*}
by using the variation condition for the first summand and Hölder's inequality for the second.  By our moment and  P-NED assumptions
\begin{equation*}
(\ref{2II.A}) \leq \sqrt{L\epsilon} + 2M^{\frac{1}{2+\delta}}(2a_l\Phi(\epsilon))\leq C\left(\sqrt{\epsilon}+ (2a_l\Phi(\epsilon))\del\right).
\end{equation*}
For $(\ref{2II.B})$ we use similar arguments:
\begin{align*}
(\ref{2II.B}) & \leq \E[\Vert h(\tilde{X}_i,X_{j,l})-h(\tilde{X}_i,X_j)\Vert^2 \textbf{1}_{\{ \Vert X_j-X_{j,l}\Vert > \epsilon\}}]\wur \\
& + \E[\Vert h(\tilde{X}_i,X_{j,l})-h(\tilde{X}_i,X_j)\Vert^2 \textbf{1}_{\{ \Vert X_j-X_{j,l}\Vert \leq \epsilon\}}]\wur \\
& \leq \E[\Vert h(\tilde{X}_i,X_{j,l})-h(\tilde{X}_i,X_j)\Vert ^{2+\delta}]^{\frac{1}{2+\delta}} \cdot \Prob(\Vert X_j-X_{j,l}\Vert > \epsilon)\del + \sqrt{L\epsilon} \\
& \leq 2M^{\frac{1}{2+\delta}}(a_l\Phi(\epsilon))\del +\sqrt{L\epsilon} \leq C\left( \sqrt{\epsilon} + (a_l\Phi(\epsilon))\del \right)
\end{align*}
Putting these two terms together, we get
\begin{align*}
(\ref{2II})  \leq C\left( \sqrt{\epsilon} + (a_l\Phi(\epsilon))\del \right).
\end{align*}
Bounding $(\ref{2I})$ works completely analogous, just with $i$ and $j$  interchanged, so
\[(\ref{2I}) \leq \left( \sqrt{\epsilon} + (a_l\Phi(\epsilon))\del \right).  \]
All together this yields
\begin{align*}
\E[\Vert h_{2,l}(X_{i,l},X_{j,l})-h_2(X_{i,l},X_{j,l}) \Vert^2]\wur \leq (\ref{2I})+(\ref{2II}) \leq C\left( \sqrt{\epsilon} + (a_l\Phi(\epsilon))\del \right) .\\
\end{align*}
So we finally get that
\begingroup
\allowdisplaybreaks
\begin{align*}
& \E\left[\left( \sum_{n_1\leq i\leq n_2,\; n_3\leq j\leq n_4} \Vert h_{2,l}(X_{i,l},X_{j,l})-h_2(X_{i,l},X_{j,l}) \Vert \right)^2\right]\wur \\
& \leq \E\left[\left( \sum_{1\leq i<j,\; n_3\leq j\leq n_4} \Vert h_{2,l}(X_{i,l},X_{j,l})-h_2(X_{i,l},X_{j,l}) \Vert \right)^2\right]\wur \\
& \leq \sum_{1\leq i<j,\; n_3\leq j\leq n_4} \E[\Vert h_{2,l}(X_{i,l},X_{j,l})-h_2(X_{i,l},X_{j,l}) \Vert ^2]\wur  \\
& \leq \sum_{1\leq i<j,\; n_3\leq j\leq n_4} C\left( \sqrt{\epsilon} + (a_l\Phi(\epsilon))\del \right)\\ 
& \leq C (n_4-n_3)\sum_{k=1}^{n_4}\left(\sqrt{\epsilon}+(a_l\Phi(\epsilon))\del \right) \leq C (n_4-n_3)n_4^{\frac{1}{4}} 
\end{align*}
where the last line is achieved by setting $\epsilon=l^{-8\frac{3+\delta}{\delta}}$ and similar calculations as in Lemma \ref{P2}.
\endgroup
\end{proof}

\begin{lem}\label{P4}
Under the assumptions of Theorem \ref{theo1}, it holds for any $n_1<n_2 <n_3 < n_4$ and $l= \left \lfloor{n_4^{\frac{3}{16}}}\right\rfloor $:
\[ \E\bigg[\Big( \big\| \sum_{n_1\leq i \leq n_2,\, n_3\leq j\leq n_4} h_{2,l}(X_{i,l},X_{j,l})\big\| \Big)^2  \bigg] \leq C (n_4-n_3)n_4^{\frac{3}{2}}. \]
For the definition of $h_{2,l}$, see Lemma \ref{P3}.
\end{lem}
\begin{proof}
In this proof, we want to use Lemma 1 \cite{Y76}, which is the following: \newline
\textit{Let $g(x_1,...,x_k)$ be a Borel function. For any $0 \leq j \leq k-1$ with}
\[ \mathit{ \E[\vert g(X_{I,l}, X'_{I^C,l}) \vert ^{1 +\tilde{\delta}} ] \leq M } \tag{$\Diamond$} \label{YCond} \] \textit{ for some $\tilde{\delta} > 0$, where $I = \{i_1,...,i_j\}$, $I^C = \{i_{j+1},...,i_k\}$ and $X'$ an independent copy of $X$, it holds that}
\[ \left\vert \E[g(X_{i_1,l},...,X_{i_k,l})]- \E[g(X_{I,l}, X'_{I^C,l})] \right\vert \leq 4 M^{1/(1+\tilde{\delta})} \beta_{(i_{j+1}-i_j)-2l}^{\tilde{\delta}/(1+\tilde{\delta})}.  \tag{Y} \label{YState} \]

Now, for the proof of the lemma, first observe that we can rewrite the squared norm as the scalar product and thus:
\begin{align}
& \E[\Vert\sum_{n_1\leq i\leq n_2,\, n_3\leq j\leq n_4} h_{2,l}(X_{i,l},X_{j,l})\Vert ^2] \nonumber \\
& = \E[\langle \sum_{n_1\leq i\leq n_2,\, n_3\leq j\leq n_4} h_{2,l}(X_{i,l},X_{j,l}), \sum_{n_1\leq i\leq n_2,\, n_3 \leq j\leq n_4} h_{2,l}(X_{i,l},X_{j,l})\rangle ] \nonumber\\
& = \underset{(i_1 \neq i_2) \,\text{or}\, (j_1 \neq j_2) \,\text{or both} }{\sum_{n_1\leq i_1 \leq n_2,\, n_3 \leq j_1\leq n_4} \sum_{n_1\leq i_2 \leq n_2,\, n_3 \leq j_2\leq n_4}} \E[\langle  h_{2,l}(X_{i_1,l},X_{j_1,l}),  h_{2,l}(X_{i_2,l},X_{j_2,l}) \rangle ] \label{Y_A}\\
& \hspace{15pt}+ \sum_{n_1\leq i \leq n_2,\, n_3 \leq j\leq n_4} \E[\langle h_{2,l}(X_{i,l},X_{j,l}), h_{2,l}(X_{i,l},X_{j,l}) \rangle]  \label{Y_B}
\end{align}

We know by the uniform moments under approximation that (\ref{Y_B}) is bounded by the following: 
\begin{align*}
 (\ref{Y_B})& = \sum_{n_1\leq i \leq n_2,\, n_3 \leq j\leq n_4} \E[\Vert h_{2,l}(X_{i,l},X_{j,l}) \Vert ^2 ] \leq (n_2-n_1)(n_4-n_3) M \\
& < n_4 (n_4-n_3) M
\end{align*}
For (\ref{Y_A}) we use the above mentioned lemma of \cite{Y76}. Note that by the double summation, we have three different cases to analyse: $(i_1\neq i_2)$ or $(j_1 \neq j_2)$ or both. Universal, let $m=\max(j_1-i_1, j_2-i_2)$, first assume that  $m=j_1-i_1$ and let $\tilde{\delta}=\delta/2 > 0$. \newline
\underline{\text{First case:}} $i_1 \neq i_2$ and $j_1 \neq j_2$ \newline
Define the function $g(x_1,x_2,x_3,x_4):=\langle h_{2,l}(x_1,x_2),h_{2,l}(x_3,x_4) \rangle$ and check that (\ref{YCond}) holds true for $I=\{i_1\}$ and $I^C = \{j_1, i_2, j_2 \}$:
\begin{align*}
& \E[\vert g(X_{i_1,l}, X'_{j_1,l}, X'_{i_2,l}, X'_{j_2,l}) \vert^{1+\tilde{\delta}}] \leq \E[\Vert h_{2,l}(X_{i_1,l}, X'_{j_1,l}) \Vert ^{1+\tilde{\delta}} \Vert h_{2,l}(X'_{i_2,l}, X'_{j_2,l}) \Vert ^{1+\tilde{\delta}}] \\
& \leq \E[\Vert h_{2,l}(X_{i_1,l}, X'_{j_1,l}) \Vert^{2(1+\tilde{\delta})}]^{1/2} \E[\Vert h_{2,l}(X'_{i_2,l}, X'_{j_2,l}) \Vert^{2(1+\tilde{\delta})}]^{1/2} \leq M 
\end{align*}
by our moment assumptions and $\delta = \tilde{\delta}/2$. Here, we first use the Cauchy-Schwarz inequality and then Hölder's inequality. Now (\ref{YState}) states that
\begin{align}\label{Yoshi1}
\left\vert \E[ g(X_{i_1,l}, X_{j_1,l}, X_{i_2,l}, X_{j_2,l}) ] - \E[g(X_{i_1,l}, X'_{j_1,l}, X'_{i_2,l}, X'_{j_2,l})] \right\vert \leq C \beta_{m-2l}^{\tilde{\delta}/(1+\tilde{\delta})}
\end{align}
The second expectation equals 0, which can be seen by using the law of the iterated expectation:
\begin{align} \label{4}
& \E[g(X_{i_1,l}, X'_{j_1,l}, X'_{i_2,l}, X'_{j_2,l})] = \E[\E[g(X_{i_1,l}, X'_{j_1,l}, X'_{i_2,l}, X'_{j_2,l})| X'_{j_1,l}, X'_{i_2,l}, X'_{j_2,l} ]] \nonumber \\
& = \E[\E[\langle h_{2,l}(X_{i_1,l}, X'_{j_1,l}), h_{2,l}(X'_{i_2,l}, X'_{j_2,l}) \rangle |   X'_{j_1,l}, X'_{i_2,l}, X'_{j_2,l} ]  ] \nonumber \\ 
& = \E[ \langle \E[ h_{2,l}(X_{i_1,l}, X'_{j_1,l}) |  X'_{j_1,l}, X'_{i_2,l}, X'_{j_2,l}  ], h_{2,l}(X'_{i_2,l}, X'_{j_2,l}) \rangle  ]
\end{align}
since $h_{2,l}(X'_{i_2,l}, X'_{j_2,l})$ is measurable with respect to the inner (conditional) expectation. In general it holds for random variables $X,Y$ that $\E[ \langle Y,X \rangle | \mathfrak{B}]=\langle Y, \E[X|\mathfrak{B}] \rangle$ if $Y$ is measurable with respect to $\mathfrak{B}$.
So, 
\begin{align*}
(\ref{4}) =  \E[ \langle \underbrace{\E[ h_{2,l}(X_{i_1,l}, X'_{j_1,l}) |  X'_{j_1,l}, X'_{i_2,l}, X'_{j_2,l}  ]}_{=\; 0 \text{ because $h_{2,l}$ is degenerated}}, h_{2,l}(X'_{i_2,l}, X'_{j_2,l}) \rangle  ] = 0.
\end{align*}
Plugging this into (\ref{Yoshi1}), we get that
\begin{align*}
\E[\langle h_{2,l}(X_{i_1,l}, X_{j_1,l}), h_{2,l}(X_{i_2,l}, X_{j_2,l})\rangle  ] \leq \left\vert \E[ g(X_{i_1,l}, X_{j_1,l}, X_{i_2,l}, X_{j_2,l})] \right\vert \leq C \beta_{m-2l}^{\tilde{\delta}/(1+\tilde{\delta})} .
\end{align*}
We repeat the above argumentation for the other two cases: \newline
\underline{Second case:} $i_1 \neq i_2$ but $j_1=j_2$ \newline
Define the function $g(x_1,x_2,x_3) := \langle h_{2,l}(x_1,x_2), h_{2,l}(x_3,x_2) \rangle$ and check that (\ref{YCond}) holds true for $I=\{i_1\}$ and $I^C=\{j_1,i_2\}$:
\begin{align*}
& \E[\vert g(X_{i_1,l},X'_{j_1,l}, X'_{j_2,l}) \vert ^{1+\tilde{\delta}}] \\
& \leq \E[\Vert h_{2,l}(X_{i_1,l}, X'_{j_1,l}) \Vert^{2(1+\tilde{\delta})}]^{1/2} \E[\Vert h_{2,l}(X'_{i_2,l}, X'_{j_1,l}) \Vert^{2(1+\tilde{\delta})}]^{1/2} \leq M 
\end{align*}
Here, (\ref{YState}) states that 
\begin{align} \label{Yoshi2}
\left\vert \E[g(X_{i_1,l}, X_{j_1,l}, X_{i_2,l})] -\E[g(X_{i_1,l}, X'_{j_1,l}, X'_{i_2,l})] \right\vert \leq C \beta_{m-2l}^{\tilde{\delta}/(1+\tilde{\delta})}
\end{align}
Again, the second expectation equals zero:
\begin{align*}
\E[g(X_{i_1,l}, X'_{j_1,l}, X'_{i_2,l})] & = \E[\E[\langle h_{2,l}(X_{i_1,l},X'_{j_1,l}), h_{2,l}(X'_{i_2,l},X'_{j_1,l}) \rangle \vert X'_{i_2,l}, X'_{j_1,l}    ]  ] \\
& = \E[ \langle \underbrace{\E[h_{2,l}(X_{i_1,l},X'_{j_1,l}) \vert X'_{i_2,l}, X'_{j_1,l}]}_{=0},   h_{2,l}(X'_{i_2,l},X'_{j_1,l}) \rangle    ] \\
& = 0  
\end{align*}
Plugging this into (\ref{Yoshi2}), we get that
\begin{align*}
\E[\langle h_{2,l}(X_{i_1,l},X_{j_1,l}), h_{2,l}(X_{i_2,l},X_{j_1,l}) \rangle  ] \leq  \left\vert\E[g(X_{i_1,l}, X_{j_1,l}, X_{i_2,l})] \right\vert \leq C \beta_{m-2l}^{\tilde{\delta}/(1+\tilde{\delta})}.
\end{align*}
\underline{Third case:} $j_1\neq j_2$ but $i_1=i_2$ \newline
Define the function $g(x_1,x_2,x_3) := \langle h_{2,l}(x_1,x_2), h_{2,l}(x_1,x_3) \rangle$. Checking that (\ref{YCond}) holds true for $I=\{i_1\}$ and $I^C=\{j_1,j_2\}$ works completely similar to the second case. And noting that we have to condition on $X_{i_1,l}, X'_{j_2,l}$ in this case, yields: 
\begin{align*}
\E[\langle h_{2,l}(X_{i_1,l},X_{j_1,l}), h_{2,l}(X_{i_1,l},X_{j_2,l}) \rangle  ] \leq \left\vert\E[g(X_{i_1,l}, X_{j_1,l}, X_{j_2,l})] \right\vert \leq C \beta_{m-2l}^{\tilde{\delta}/(1+\tilde{\delta})}
\end{align*}

We can conclude for the quadratic term:
\begin{align}\label{5}
& \E[\Vert \sum_{n_1\leq i \leq n_2,\, n_3 \leq j\leq n_4} h_{2,l}(X_{i,l},X_{j,l})\Vert ^2 ]= \nonumber \\
& \underset{(i_1 \neq i_2) \,\text{or}\, (j_1 \neq j_2) \,\text{or both} }{\sum_{n_1\leq i_1 \leq n_2,\, n_3 \leq j_1\leq n_4} \sum_{n_1\leq i_2\leq n_2,\, n_3\leq j_2\leq n_4}} C \beta_{m-2l}^{\tilde{\delta}/(1+\tilde{\delta})} + n_4(n_4-n_3)M
\end{align}

For a fixed $m$ we have the following possibilities to choose:\\
Since we assumed $m=j_1-i_1$, there are \begin{itemize}
\item at most $n_2-n_1 < n_4 $ possibilities for $i_1$, so only $1$ possibility for $j_1$
\item at most $(n_4-n_3)$ possibilities for $j_2$, so at most $m$ possibilities for $i_2$, since by the definition of $m$ the value $j_2-i_2$ is smaller (or equal) than $m$.
\end{itemize}
So, recalling that $\delta = \tilde{\delta}/2$, we have
\begin{align*}
&\underset{(i_1 \neq i_2) \,\text{or}\, (j_1 \neq j_2) \,\text{or both} }{\sum_{n_1\leq i_1\leq n_2,\, n_3\leq j_1\leq n_4} \sum_{n_1\leq i_2\leq n_2,\, n_3\leq j_2\leq n_4}} C \beta_{m-2l}^{\tilde{\delta}/(1+\tilde{\delta})}\\
& \leq C (n_4-n_3)n_4 \sum_{m=1}^{n_4} m \beta_{m-2l}^{\frac{\delta}{2+\delta}} = C (n_4-n_3) \left( \sum_{m=1}^{2l-1}m \underbrace{\beta_{m-2l}^{\frac{\delta}{2+\delta}}}_{=1} + \sum_{m=2l}^{n_4} \beta_{m-2l}^{\frac{\delta}{2+\delta}} \right)\displaybreak[0] \\
& \leq  C (n_4-n_3)n_4 \left( \sum_{m=1}^{2l-1} m + \sum_{m=2l}^{n_4} (m-2l)\bm + \sum_{m=2l}^{n_4} 2l\bm \right) \\
&  \leq C (n_4-n_3)n_4 \left( (2l)^2 + \sum_{m=2l}^{n_4} (m-2l)\bm +2l \sum_{m=2l}^{n_4}(m-2l)\bm \right) \\
& = C (n_4-n_3)n_4 \left( l^2 + (1+2l)\sum_{m=2l}^{n_4}(m-2l)\bm \right) \displaybreak[0]\\
& \leq C (n_4-n_3)n_4 \left( l^2 + (2l)^2\sum_{m=2l}^{n_4} (m-2l) \bm \right) \;\;\; \text{for $l>\frac{1}{2}$}\\
& \leq C (n_4-n_3)n_4 \Big( l^2 + l^2 \underbrace{\sum_{m=2l}^{n_4}(m-2l)^2 \bm}_{<\infty } \Big) \\
& \leq C  (n_4-n_3) n_4 l^2  \leq C (n_4-n_3) n_4^{\frac{3}{2}} .
\end{align*}
So $(\ref{5})  \leq  C (n_4-n_3) n_4^{\frac{3}{2}}$. If $m=j_2 -i_2$, it works very similar. Just a few comments on what changes: We get in the first case $I=\{i_1,j_1,j_2\},\; I^C = \{j_2\}$, which leads to defining the function $g(X_{i_1,l},X_{j_1,l}, X_{i_2,l}, X'_{j_2,l} ) := \langle h_{2,l}(X_{i_1,l},X_{j_1,l}), h_{2,l}(X_{i_2,l}, X'_{j_2,l} )\rangle$ and conditioning on $X_{i_1,l},X_{j_1,l}, X_{i_2,l}$. For the second case it is $I = \{i_1,i_2\},\; I^C =\{j_2\}$. We define $g(X_{i_1,l}, X'_{j_2,l}, X_{i_2,l}) := \langle h_{2,l}(X_{i_1,l}, X'_{j_2,l}), h_{2,l}(X_{i_2,l}, X'_{j_2,l}) \rangle$ and condition on $X_{i_2,l}, X'_{j_2,l}$. In the third case it is $I=\{i_1,j_1\},\; I^C=\{j_2\}$, function $g(X_{i_1,l}, X_{j_1,l}, X'_{j_2,l}) := \langle h_{2,l}(X_{i_1,l}, X_{j_1,l}), h_{2,l}(X_{i_1,l}, X'_{j_2,l}) \rangle$ and we condition on $X_{i_1,l},X_{j_1,l}$.\\
This proves the lemma.
\end{proof}

\begin{prop} \label{degThm}
Under the assumptions of Theorem \ref{theo1}, it holds that 
\begin{itemize}
\item[a)]  \begin{align*}
E\bigg[\Big( \max_{1 \leq n_1 < n} \big\|\sum_{i=1}^{n_1}\sum_{j=n_1+1}^n h_2(X_i,X_j) \big\| \Big)^2  \bigg]\wur \leq C s^2 2^{\frac{5s}{4}} 
\\ \text{for $s$ large enough that $n\leq 2^s$.}
\end{align*}
\item[b)] \[ \max_{1\leq n_1 < n} \frac{1}{n^{3/2}} \Big\| \sum_{i=1}^{n_1}\sum_{j=n_1+1}^n h_2(X_i,X_j) \Big\| \xrightarrow{\text{a.s.}} 0 \;\;\; \text{for } n\to\infty . \]
\end{itemize}
\end{prop}

\begin{proof}
\underline{Part a)}
We split the expectation with the help of the triangle inequality into three parts:
\begin{align}
& \E\bigg[\Big( \max_{1 \leq n_1 < n} \Vert \sum_{i=1}^{n_1}\sum_{j=n_1+1}^n h_2(X_i,X_j) \Vert \Big)^2  \bigg]\wur \nonumber \\
& \leq \E\bigg[ \Big( \max_{1 \leq n_1 < n} \sum_{i=1}^{n_1} \sum_{j=n_1+1}^n \Vert h_2(X_i,X_j)-h_2(X_{i,l},X_{j,l}) \Vert \Big)^2   \bigg]\wur  \label{3I} \\
& + \E\bigg[\Big( \max_{1 \leq n_1 < n} \sum_{i=1}^{n_1} \sum_{j=n_1+1}^n \Vert h_{2,l}(X_{i,l},X_{j,l})-h_2(X_{i,l},X_{j,l}) \Vert \Big)^2  \bigg]\wur \label{3II}  \\
& + \E\bigg[\Big( \max_{1 \leq n_1 < n} \Vert \sum_{i=1}^{n_1} \sum_{j=n_1+1}^n h_{2,l}(X_{i,l},X_{j,l}) \Vert \Big)^2  \bigg]\wur  \label{3III}
\end{align}
We want to use Lemma \ref{P2}-\ref{P4} to bound the three terms. Because the summands of (\ref{3I}) are all positive, we have by Lemma \ref{P2}
\begin{equation*}
 (\ref{3I})  \leq \E\bigg[ \Big(\sum_{j=1}^{n} \sum_{i=1}^{j-1} \Vert h_2(X_i,X_j)-h_2(X_{i,l},X_{j,l}) \Vert \Big)^2   \bigg]\leq Cn^{5/4}.
\end{equation*}
(\ref{3II}) can be bounded in the same way, using Lemma \ref{P3}. For (\ref{3III}), the idea is to rewrite the double sum. First note that for $n_1<n_2$
\begin{align*}
&\sum_{i=1}^{n_2} \sum_{j=n_2+1}^n h_{2,l}(X_{i,l},X_{j,l})-\sum_{i=1}^{n_1} \sum_{j=n_1+1}^n h_{2,l}(X_{i,l},X_{j,l})\\
=&\sum_{i=n_1+1}^{n_2} \sum_{j=n_2+1}^n h_{2,l}(X_{i,l},X_{j,l})-\sum_{i=1}^{n_1} \sum_{j=n_1+1}^{n_2} h_{2,l}(X_{i,l},X_{j,l}).
\end{align*}
So we can conclude by Lemma \ref{P4} that
\begin{align*}
& \E\bigg[\Big( \Vert \sum_{i=1}^{n_2} \sum_{j=n_2+1}^n h_{2,l}(X_{i,l},X_{j,l})-\sum_{i=1}^{n_1} \sum_{j=n_1+1}^n h_{2,l}(X_{i,l},X_{j,l}) \Vert \Big)^2\bigg]\\
& \leq (n_2-n_1)n^{3/2}\leq (n_2-n_1)2^{3s/2}
\end{align*}
as $n\leq 2^s$. By Theorem 1 \cite{M76} (which also holds in Hilbert spaces) it follows that
\begin{equation*}
\E\bigg[\Big( \max_{1 \leq n_1 < n} \Vert \sum_{i=1}^{n_1} \sum_{j=n_1+1}^n h_{2,l}(X_{i,l},X_{j,l}) \Vert \Big)^2  \bigg] \leq Cs^2 2^{5s/2}
\end{equation*}
and by taking the square root
\begin{align*}
 (\ref{3III}) = \E\bigg[\Big( \max_{1 \leq n_1 < n} \Vert \sum_{i=1}^{n_1} \sum_{j=n_1+1}^n h_{2,l}(X_{i,l},X_{j,l}) \Vert \Big)^2  \bigg]\wur \leq C s^{\frac{3}{2}}2^{\frac{5s}{4}} \leq Cs^2 2^{\frac{5s}{4}}.
\end{align*}
This yields all together
\[\E\bigg[\Big( \max_{1 \leq n_1 < n} \Vert \sum_{i=1}^{n_1}\sum_{j=n_1+1}^n h_2(X_i,X_j) \Vert \Big)^2  \bigg]\wur  \leq Cs^2 2^{\frac{5s}{4}}  \]

\underline{Part b)}
Recall that $s$ is chosen such that $n\leq 2^s$ and thus $n^{\frac{3}{2}} \leq 2^{\frac{3s}{2}}$. To prove almost sure convergence, it is enough to prove that for any $\epsilon>0$
\[ \sum_{s=1}^{\infty}\Prob\Big( 2^{-\frac{3s}{2}} \max_{1\leq n_1<n} \big\| \sum_{s=1}^{n_1}\sum_{j=n_1+1}^n h_2(X_i,X_j) \big\|  > \epsilon\Big) < \infty  \]
We do this by using Markov's inequality and our result from a):
\begin{align*}
& \sum_{s=1}^{\infty}\Prob\Big( 2^{-\frac{3s}{2}} \max_{1\leq n_1<n} \Vert \sum_{s=1}^{n_1}\sum_{j=n_1+1}^n h_2(X_i,X_j) \Vert  > \epsilon\Big) \\
& \leq \frac{1}{\epsilon^2}\sum_{s=1}^{\infty} \E\bigg[\Big(2^{-\frac{3s}{2}}\max_{1\leq n_1<n} \Vert \sum_{s=1}^{n_1}\sum_{j=n_1+1}^n h_2(X_i,X_j) \Vert \Big)^2  \bigg] \\
& = \frac{1}{\epsilon^2}\sum_{s=1}^{\infty} 2^{-3s} \E\bigg[\Big(\max_{1\leq n_1<n} \Vert \sum_{s=1}^{n_1}\sum_{j=n_1+1}^n h_2(X_i,X_j) \Vert \Big)^2  \bigg] \\
& \leq \frac{1}{\epsilon^2} \sum_{s=1}^{\infty} 2^{-3s}(Cs^2 2^{\frac{5s}{4}})^2 \;\;\; \text{by part a)} \\
& = \frac{C}{\epsilon^2} \sum_{s=1}^{\infty} s^4 2^{-\frac{s}{2}} < \infty
\end{align*}
By the Borel-Cantelli lemma follows the almost sure convergence
\[ \max_{1\leq n_1 < n} \frac{1}{n^{3/2}} \Vert \sum_{i=1}^{n_1}\sum_{j=n_1+1}^n h_2(X_i,X_j) \Vert \xrightarrow{\text{a.s.}} 0 \;\;\; \text{for } n\to\infty . \]
\end{proof}

\subsection{Results under Alternative}

Recall our model under the alternative: \\
$(X_n, Z_n)_{n\in\Z}$ is a stationary, $H\otimes H$-valued sequence and we observe $Y_1,...,Y_n$ with
\begin{equation*}
Y_i=\begin{cases}X_i \  \ &\text{for}\ i\leq \lfloor n\lambda^\star \rfloor = k^\star \\ Z_i \  \ &\text{for}\ i> \lfloor n\lambda^\star \rfloor = k^\star \end{cases},
\end{equation*}
so $\lambda^\star\in(0,1)$ is the proportion of observations after which the change happens. We assume that the process $(X_i,Z_i)_{i\in\mathbb{Z}}$ is stationary and P-NED on an absolutely regular sequences $(\zeta_n)_{n\in \mathbb{Z}}$. \\
Let $h: H \times H \rightarrow H$ be an antisymmetric kernel and assume that $\E[h(X_0,\tilde{Z}_0)] \neq 0$, where $\tilde{Z}_0$ is an independent copy of $Z_0$ and independent of $X_0$. Since $X_0$ and $\tilde{Z}_0$ are not identically distributed, Hoeffding's decomposition of $h$ equals
\[h(x,y) = h_1^\star(x)-h_1(y)+h_2^\star(x,y) \]
where
\begin{align} h_1(x)&=\E[h(x,X_0)]\; , \;\;\;\;\; h_1^\star(x)=\E[h(x,Z_0)]  \\
\label{degalt}h_2^\star(x,y)&= h(x,y)-h_1^\star(x)+h_1(y)
\end{align}
So it holds for the test statistic $U_{n,k^\star}(Y):= \sum_{i=1}^{k^\star}\sum_{j=k^\star+1}^n h(Y_i,Y_j)$ that
\begin{align*}
U_{n,k^\star}(Y) & = \sum_{i=1}^{k^\star}\sum_{j=k^\star+1}^n h(X_i,Z_j) \\
& = \sum_{i=1}^{k^\star}\sum_{j=k^\star+1}^n\big( h_1^\star(X_i)-h_1(Z_j) +h_2^\star(X_i,Z_j)\big) \\
& = (n-k^\star) \sum_{i=1}^{k^\star} h_1^\star(X_i) - k^\star\sum_{j=k^\star+1}^n h_1(Z_j) + \sum_{i=1}^{k^\star}\sum_{j=k^\star+1}^n h_2^\star(X_i,Z_j).
\end{align*}

\begin{lem} \label{A_deg}
Let the Assumption of Theorem \ref{theo2} hold for $(X_i,Z_i)_{i\in \mathbb{Z}}$ and let $h_2^\star$ as defined in \eqref{degalt}. Then it holds that
\[  \frac{1}{n^{3/2}}\Vert \sum_{i=1}^{k^\star} \sum_{j=k^\star+1}^n h_2^\star(X_i,Z_j)+\E[h(X_0,\tilde{Z}_0) \Vert  \xrightarrow{\text{a.s.}} 0 \;\;\; \text{for } n\to\infty,  \]
where $\tilde{Z}_0$ is an independent copy of $Z_0$ and independent of $X_0$.
\end{lem}
\begin{proof}
Notice that $h_2^\star(x,z)+ \E[h(X_0,\tilde{Z}_0)$ is degenerated since $\E[h_1^\star(X_0)]= \linebreak\E[h(X_0,\tilde{Z}_0)]$ and
\begin{multline*}
E\big[ h_2^\star(X_0,z)+ \E[h(X_0,\tilde{Z}_0)] \big]=E\big[ h(X_0,y)-h_1^\star(X_0)-h_1(y)+\E[h(X_0,\tilde{Z}_0)] \big]\\
=h_1(y)-\E[h(X_0,\tilde{Z}_0)]-h_1(y)+\E[h(X_0,\tilde{Z}_0)]=0
\end{multline*}
and similarly $ \E[h_2^\star(x,\tilde{Z}_0)+ \E[h(X_0,\tilde{Z}_0)]  = 0  $. So we can prove the lemma along the same arguments as under the null hypothesis.  
\end{proof}

\begin{lem} \label{A_lin}
Under the assumption of Theorem \ref{theo2} it holds that 
\[ \Big( \frac{1}{\sqrt{n}} \sum_{i=1}^{\left \lfloor{n\lambda}\right\rfloor  }\big( h_1^\star(X_i)-\E[h(X_0,\tilde{Z}_0)]\big) \Big)_{\lambda\in [0,1]} \Rightarrow (W_1(\lambda))_{\lambda\in [0,1]}  \] and
\[ \Big( \frac{1}{\sqrt{n}} \sum_{i=1}^{\left \lfloor{n\lambda}\right\rfloor  }\big( h_1(Z_i)+\E[h(X_0,\tilde{Z}_0)]\big)\Big)_{\lambda\in [0,1]} \Rightarrow (W_2(\lambda))_{\lambda\in [0,1]}  \] 
where $(W_1(\lambda))_{\lambda\in [0,1]}$, $(W_2(\lambda))_{\lambda\in [0,1]}$ are Brownian motions with covariance operator as defined in Theorem \ref{theo1}.
\end{lem}
\begin{proof}
The proof follows the steps of Theorem \ref{linThm}. So, we have to check the assumptions of Theorem 1 \cite{STW16}. We will do this for $h^\star_1(X_i)$, for $h_1(Z_i)$ everything holds similarly. First note that  $\E[h_1^\star(X_0)]=\E[h(X_0,\tilde{Z}_0)]$.

\underline{Assumption 1:} $(h_1^\star(X_n))_{n \in\mathbb{Z}}$ is $L_1$-NED.

Along the lines of the proof of Lemma \ref{linLem} we can show that $(h_1^\star(X_n))_{n\in\mathbb{Z}}$ is $L_2$-NED with approximating constants $a_{k,2}= \mathcal{O}( k^{-4\frac{3+\delta}{\delta}})$. By Jensen's inequality it follows that $(h_1^\star(X_n))_{n\in\mathbb{Z}}$ is $L_1$-NED with approximating constants $a_{k,1}=a_{k,2}$. 

\underline{Assumption 2:} Existing $(4+\delta)$-moments.

Recall that $h_1^\star(x)=\E[h(x,\tilde{Z}_0)]$, so by Jensen inequality
\begin{equation*}
E\left[|h_1^\star(X_i)|^{4+\delta}\right]\leq E[|h(X_1,\tilde{Z}_1)|^{4+\delta}]<\infty
\end{equation*}

\underline{Assumption 3:} $\sum_{m=1}^{\infty} m^2 a_{m,1}^{\frac{\delta}{3+\delta}} \leq \infty$  follows similar as in Theorem \ref{linThm}.

\underline{Assumption 4:} $\sum_{m=1}^{\infty} m^2 \beta_m^{\frac{\delta}{4+\delta}} < \infty$ is assumed in Theorem \ref{theo2}.

\end{proof}

\begin{cor} \label{A_cor}
Under assumptions of Theorem \ref{theo1}, it holds that
\[\frac{1}{n^{3/2}} \sum_{i=1}^{k^\star}\sum_{j=k^\star +1}^n \big(h_1^\star(X_i)-h_1(Z_j)-2\E[h(X_0,\tilde{Z}_0)]\big) \]
is stochastically bounded.
\end{cor}

\begin{proof}
This follows from Lemma \ref{A_lin} above:
\begin{align*}
& \bigg|\frac{1}{n^{3/2}} \sum_{i=1}^{k^\star}\sum_{j=k^\star+1}^n\big( h_1^\star(X_i)-h_1(Z_j)-2\E[h(X_0,\tilde{Z}_0)]\big)\bigg| \\
\leq &\bigg|\frac{1}{n^{1/2}} \sum_{i=1}^{k^\star} h_1^\star(X_i)-\E[h(X_0,\tilde{Z}_0)]\bigg| +\bigg| \frac{1}{n^{1/2}} \sum_{j=k^\star+1}^n h_1(Z_j)+\E[h(X_0,\tilde{Z}_0)] \bigg|
\end{align*}
Both summands converge weakly to a Gaussian limit and are stochastically bounded.
\end{proof}

\subsection{Dependent Wild Bootstrap}

\begin{prop}\label{degPart}
Let $(\varepsilon_i)_{i\leq n, n \in \mathbb{N}}$ be a triangular scheme of random multiplier independent from $(X_i)_{i \in \mathbb{Z}}$, such that the moment condition  $ \E[| \varepsilon_i | ^2] < \infty  $ holds.  \\
Then under the Assumptions of Theorem \ref{theo1}, it holds that 
\[ \max_{1\leq k < n} \frac{1}{n^{3/2}} \big\| \sum_{i=1}^{k}\sum_{j=k+1}^n h_2(X_i,X_j)(\varepsilon_i+\varepsilon_j) \big\| \xrightarrow{\text{a.s.}} 0 \;\;\; \text{for } n\to\infty  \]
\end{prop}

\begin{proof} The statement follows along the line of the proofs of the Lemmas \ref{P3} to  \ref{P4} and Proposition \ref{degThm}. For this, note that by the independence of $(\varepsilon_i)_{i\leq n, n \in \mathbb{N}}$ and  $(X_i)_{i \in \mathbb{Z}}$ and by Lemma \ref{P1}
\begin{align*}
& \E[\Vert h_2(X_i,X_{i+k+2l})(\varepsilon_{i}+\varepsilon_{i+k+2l})-h_2(X_{i,l},X_{i+k+2l,l}(\varepsilon_{i}+\varepsilon_{i+k+2l}))\Vert^2]^{\frac{1}{2}} \\
&= \E[\Vert h_2(X_i,X_{i+k+2l})-h_2(X_{i,l},X_{i+k+2l,l}\Vert^2]^{\frac{1}{2}} \cdot\E[(\varepsilon_{i}+\varepsilon_{i+k+2l})^2]^{\frac{1}{2}} \\
& \leq C(\sqrt{\epsilon}+\beta_k^{\frac{\delta}{2(2+\delta)}}+(a_l\Phi(\epsilon))^{\frac{\delta}{2(2+\delta)}} ) 
\end{align*}
From this, we can conclude that for any $n_1 < n_2 < n_3 < n_4$ and $l= \left \lfloor{n_4^{\frac{3}{16}}}\right\rfloor $:
\[\E\bigg[\Big(\sum_{n_1 \leq i \leq n_2}\sum_{n_3 \leq j \leq n_4} \Vert h_2(X_i,X_j)-h_2(X_{i,l},X_{j,l})(\varepsilon_{i}+\varepsilon_{j})\Vert \Big)^2 \bigg]\wur \leq C(n_4-n_3)n_4^{\frac{1}{4}} \] 
as in Lemma \ref{P2}. Similary, we obtain (making use of the independence of $(\varepsilon_i)_{i\leq n}$ and  $(X_i)_{i \in \mathbb{Z}}$ again)
\begin{align*}
&\E[\Vert h_{2,l}(X_{i,l},X_{j,l})-h_2(X_{i,l},X_{j,l})(\varepsilon_i+\varepsilon_j) \Vert^2]\\
=&\E[\Vert h_{2,l}(X_{i,l},X_{j,l})-h_2(X_{i,l},X_{j,l}) \Vert^2] \E[(\varepsilon_i+\varepsilon_j) ^2]  \leq C\left( \sqrt{\epsilon} + (a_l\Phi(\epsilon))\del \right) 
\end{align*}
and along the lines of the proof of Lemma \ref{P3}  for any $n_1<n_2<n_3<n_4$ and $l= \left \lfloor{n_4^{\frac{3}{16}}}\right\rfloor $: 
\[ \E\bigg[\Big( \sum_{n_1\leq i \leq n_2,\, n_3\leq j\leq n_4} \Vert h_{2,l}(X_{i,l},X_{j,l})-h_2(X_{i,l},X_{j,l}) (\varepsilon_i+\varepsilon_j)\Vert \Big)^2 \bigg]\wur \!\! \leq \! C (n_4-n_3) n_4^{\frac{1}{4}}.  \]
With the same type of argument, we also obtain the analogous result to Lemma \ref{P4}:
\[ \E\bigg[\Big( \Vert \sum_{n_1\leq i \leq n_2,\, n_3\leq j\leq n_4} h_{2,l}(X_{i,l},X_{j,l})(\varepsilon_i+\varepsilon_j)\Vert \Big)^2  \bigg] \leq C (n_4-n_3)n_4^{\frac{3}{2}} \]
and then we can proceed as in the proof of Proposition \ref{degThm}.
\end{proof}

\begin{lem}\label{bootvar} Under the assumptions of Theorem \ref{theo3}, for any $t_0=0<t_1<t_2,...,t_k=1$ and any $a_1,...,a_k\in H$
\begin{equation*}
\Var\bigg[\frac{1}{\sqrt{n}}\sum_{j=1}^k\sum_{i=\lfloor
 nt_{j-1}\rfloor+1}^{\lfloor nt_j\rfloor}\langle a_j,h_1(X_i)\varepsilon_i\rangle\Big| X_1,...,X_n\bigg]\xrightarrow{\mathcal{P}}\Var\bigg[\sum_{j=1}^k\langle a_j,W(t_j)-W(t_{j-1})\rangle\bigg]
\end{equation*}

\end{lem}

\begin{proof} To simplify the notation, we introduce a triangular scheme $V_{i,n} =\langle a_j,h_1(X_i)\rangle$ for $i=\lfloor nt_{j-1}\rfloor+1,...,i=\lfloor nt_{j}\rfloor$. By our assumptions, $\operatorname{Cov}(\varepsilon_i,\varepsilon_j)=w(|i-j|/q_n)$, so we obtain for the variance condition on $X_1,...,X_n$:
\begin{align*}
& \Var\bigg[\frac{1}{\sqrt{n}}\sum_{j=1}^k\sum_{i=\lfloor nt_{j-1}\rfloor+1}^{\lfloor nt_j\rfloor}\langle a_j,h_1(X_i)\varepsilon_i\rangle\Big| X_1,...,X_n\bigg]\\
&=\sum_{i=1}^n\sum_{l=1}^nV_{i,n}V_{l,n}\operatorname{Cov}(\varepsilon_i,\varepsilon_l)=\sum_{i=1}^n\sum_{l=1}^nV_{i,n}V_{l,n}w(|i-l|/q_n).
\end{align*}
This is the kernel estimator for the variance, which is consistent even for heteroscedastic time series under the assumptions of \cite{de2000consistency}. The $L_2$-NED follows by Lemma \ref{linLem}. Note that the mixing coefficients for absolute regularity are larger than the strong mixing coefficients used by \cite{de2000consistency}, so their mixing assumption follows directly from ours.
\end{proof}

\begin{prop}\label{bootpartial}  Under the assumptions of Theorem \ref{theo3}, we have the weak convergence (in the space $D_{H^2}[0,1]$)
\begin{equation*}
\bigg(\frac{1}{\sqrt{n}}\sum_{i=1}^{[nt]}(h_1(X_i),h_1(X_i)\varepsilon_i)\bigg)_{t\in[0,1]}\Rightarrow (W(t), W^\star(t))_{t\in[0,1]}
\end{equation*}
where $W$ and $W^\star$ are independent Brownian motions with covariance operator as in Theorem \ref{theo1}.
\end{prop}

\begin{proof} We have to prove finite-dimensional convergence and tightness. As the tightness for the first component was already established in the proof of Theorem 1 of \cite{STW16}, we only have to deal with the second component. The tightness of the partial sum process of $h_1(X_i)\varepsilon_i$, $i\in\N$, can be shown along the lines of the proof of the same theorem: For this note that by the independence of $(\varepsilon_i)_{i\leq n}$ and $X_1,...,X_n$
\begin{multline*}
\left|E\left[\langle h_1(X_i)\varepsilon_i,h_1(X_j)\varepsilon_j\rangle\langle h_1(X_k)\varepsilon_k,h_1(X_l)\varepsilon_l\rangle\right]\right|\\
=\left|E\left[\langle h_1(X_i),h_1(X_j)\rangle\langle h_1(X_k),h_1(X_l)\rangle\right]E[\varepsilon_i\varepsilon_j\varepsilon_k\varepsilon_l]\right|\\
\leq 3\left|E\left[\langle h_1(X_i),h_1(X_j)\rangle\langle h_1(X_k),h_1(X_l)\rangle\right]\right|, 
\end{multline*}
the rest follows as in Lemma 2.24 of \cite{BBD01} and in the proof of Theorem 1 of \cite{STW16}.

For the finite dimensional convergence, we will show the weak convergence of the second component conditional on $h_1(X_i)\varepsilon_i$, $i\in\N$, because the weak convergence of the first component is already established in Proposition \ref{linThm}. By the continuity of the limit process, it is sufficient to study the distribution for $t_1,..,t_k\in\mathbb{Q}\cap [0,1]$ and by the Cram\'er-Wold-device and the separability of $H$, it is enough to show the convergence of the condition distribution of $\frac{1}{\sqrt{n}}\sum_{j=1}^k\sum_{i=[nt_{j-1}]+1}^{[nt_j]}\langle a_j,h_1(X_i)\varepsilon_i\rangle$ for $a_1,...,a_k$ from a countable subset of $H$. Conditional on $X_1,...,X_n$, the distribution of  $\frac{1}{\sqrt{n}}\sum_{j=1}^k\sum_{i=[nt_{j-1}]+1}^{[nt_j]}\langle a_j,h_1(X_i)\varepsilon_i\rangle$ is Gaussian with expectation 0 and variance converging to the right limit in probability by Lemma \ref{bootvar}.

Using a well-known characterization of convergence in probability, for every subseries there is another subseries such that this convergence holds almost surely. So we can construct a subseries that the almost sure convergence holds for all $k$, $t_1,..,t_k\in\mathbb{Q}\cap [0,1]$ and all $a_1,...,a_k$ from the countable subset of $H$, so we can find a subseries such that the convergence of the finite-dimensional distributions holds almost surely. Thus, the finite-dimensional convergence of the conditional distribution holds in probability and the statement of the proposition is proved.
\end{proof}

\section{Proof of Main Results}\label{Proof Main}

\begin{proof}[Proof of Theorem \ref{theo1}]
We will bound the maximum from above by the sum of the degenerate and the linear part, using Hoeffding's decomposition, as shown in Lemma \ref{Hoeff}:
\begin{align*}
& \max_{1\leq k<n} \frac{1}{n^{3/2}} \Vert U_{n,k} \Vert   = \max_{1\leq k<n} \frac{1}{n^{3/2}} \Vert n \sum_{i=1}^k(h_1(X_i)-\overline{h_1(X)})+\sum_{i=1}^k\sum_{j=k+1}^n h_2(X_i,X_j) \Vert  \\& \leq  \max_{1\leq k<n} \frac{1}{n^{3/2}} \Vert n \sum_{i=1}^k(h_1(X_i)-\overline{h_1(X)})\Vert + \max_{1\leq k<n} \frac{1}{n^{3/2}} \Vert\sum_{i=1}^k\sum_{j=k+1}^n h_2(X_i,X_j) \Vert 
\end{align*}
by triangle inequality. For the degenerate part, we can use the convergence to 0 from Proposition \ref{degThm}:
\[ \max_{1\leq k<n} \frac{1}{n^{3/2}} \Vert\sum_{i=1}^k\sum_{j=k+1}^n h_2(X_i,X_j) \Vert \xrightarrow{P} 0  \]
since convergence in probability follows from almost sure convergence.\\
Now observe that we can write the linear part as
\begin{align*}
& \max_{1\leq k<n} \frac{1}{n^{3/2}} \Vert n \sum_{i=1}^k(h_1(X_i)-\overline{h_1(X)})\Vert =  \max_{\lambda\in [0,1]} \frac{1}{n^{3/2}} \Vert n \sum_{i=1}^{\left \lfloor{n\lambda}\right\rfloor}(h_1(X_i)-\overline{h_1(X)})\Vert \\
& = \max_{\lambda\in [0,1]} \frac{1}{n^{3/2}} \Vert n  \sum_{i=1}^{\left \lfloor{n\lambda}\right\rfloor}h_1(X_i) -n \left \lfloor{n\lambda}\right\rfloor \frac{1}{n}\sum_{j=1}^n h_1(X_j) \Vert \\
& = \max_{\lambda \in [0,1]} \Vert \frac{1}{\sqrt{n}} \sum_{i=1}^{\left \lfloor{n\lambda}\right\rfloor} h_1(X_i) - \frac{\left \lfloor{n\lambda}\right\rfloor}{n^{3/2}} \sum_{j=1}^n h_1(X_j) \Vert\displaybreak[0]  \\
& \approx  \sup_{\lambda\in[0,1]} \Vert \underbrace{\frac{1}{\sqrt{n}} \sum_{i=1}^{\left \lfloor{n\lambda}\right\rfloor} h_1(X_i)}_{=: x(\lambda)} -\frac{\lambda}{\sqrt{n}} \sum_{j=1}^n h_1(X_i) \Vert \;\;\; \text{for $n$ large enough}\\
& = \sup_{\lambda\in[0,1]} \Vert x(\lambda)- \lambda x(1) \Vert
\end{align*}
We know by Proposition \ref{linThm} that 
\[ (x(\lambda))_{\lambda \in [0,1]} \xrightarrow{\mathcal{D}} (W(\lambda))_{\lambda\in[0,1]}   \]
By the continuous mapping theorem it follows that $(x(\lambda)-\lambda x(1))_{\lambda \in[0,1]} \xrightarrow{\mathcal{D}} (W(\lambda)-\lambda W(1))_{\lambda\in[0,1]}$.
And thus we can finally conclude that
\[\max_{1\leq k<n} \frac{1}{n^{3/2}} \Vert U_{n,k} \Vert \xrightarrow{\mathcal{D}} \sup_{\lambda \in[0,1]} \Vert W(\lambda)-\lambda W(1) \Vert.   \]
\end{proof}

\begin{proof}[Proof of Theorem \ref{theo2}] We can bound the maximum from below using the reverse triangle inequality and then make use of previous results:
\begin{align*}
& \max_{1\leq k\leq n} \Vert \frac{1}{n^{3/2}} U_{n,k}(Y) \Vert  \geq \Vert \frac{1}{n^{3/2}} U_{n,k^\star}(Y) \Vert \;\;\;\;\;\;\text{where $k^\star=\left \lfloor{n\lambda^\star}\right\rfloor$}\\
& =  \Vert \frac{1}{n^{3/2}} \big( U_{n,k^\star}(Y) -k^\star(n-k^\star) \E[h(X_0,\tilde{Z}_0)] \big) + \frac{k^\star(n-k^\star)}{n^{3/2}}  \E[h(X_0,\tilde{Z}_0)]\Vert  \\
& \geq \Big\vert \Vert \frac{1}{n^{3/2}} (  U_{n,k^\star}(Y) -k^\star(n-k^\star) \E[h(X_0,\tilde{Z}_0)] )\Vert - \Vert \frac{k^\star(n-k^\star)}{n^{3/2}}  \E[h(X_0,\tilde{Z}_0)]\Vert \Big\vert \\
& \;\;\; \text{by using the reverse triangle inequality} \\
& = \Big\vert \Vert \frac{1}{n^{3/2}} \sum_{i=1}^{k^\star} \sum_{j=k^\star+1}^n \big(h_1^\star(X_i)-h_1(Z_j) + h _2(X_i,Z_j) - \E[h_2(X_i,Z_j)]\big) \Vert \\
& \hspace{70pt} - \Vert \frac{k^\star(n-k^\star)}{n^{3/2}}  \E[h(X_0,\tilde{Z}_0)]\Vert \Big\vert\\
& \geq \Vert \frac{k^\star(n-k^\star)}{n^{3/2}}  \E[h(X_0,\tilde{Z}_0)]\Vert
- \Vert  \frac{1}{n^{3/2}} \sum_{i=1}^{k^\star} \sum_{j=k^\star+1}^n\big( h_1^\star(X_i)-h_1(Z_j) -2 \E[h(X_0,\tilde{Z}_0)]\big) \Vert \\
& \hspace{70pt} - \Vert \frac{1}{n^{3/2}} \sum_{i=1}^{k^\star} \sum_{j=k^\star+1}^n h _2(X_i,Z_j) +\E[h(X_0,\tilde{Z}_0)] \Vert 
\end{align*}
by using the reverse triangle inequality again. By Corollary \ref{A_cor} we know that
\[ \Vert  \frac{1}{n^{3/2}} \sum_{i=1}^{k^\star} \sum_{j=k^\star+1}^n h_1^\star(X_i)-h_1(Z_j) - 2\E[h_1^\star(X_i)-h_1(Z_j)] \Vert  \]
is stochastically bounded.
And by Lemma \ref{A_deg} it holds that
\[\Vert \frac{1}{n^{3/2}} \sum_{i=1}^{k^\star} \sum_{j=k^\star+1}^n h _2(X_i,Z_j) + \E[h_2(X_i,Z_j)] \Vert \xrightarrow{n \to \infty} 0 \]
But since $\E[h(X_0,\tilde{Z}_0)] \neq 0$ the last part diverges to infinity:
\[\Vert \frac{1}{n^{3/2}} k^\star(n-k^\star) \E[h(X_0,\tilde{Z}_0)] \Vert \approx \Vert \sqrt{n} \lambda^\star(1-\lambda^\star)  \E[h(X_0,\tilde{Z}_0)] \Vert \xrightarrow{n \to \infty} \infty,  \]
and thus $\max\limits_{1\leq k\leq n} \Vert \frac{1}{n^{3/2}} U_{n,k}(Y) \Vert \xrightarrow{n \to \infty} \infty$. 
\end{proof}

\begin{proof}[Proof of Theorem \ref{theo3}] Because the convergence in distribution of $\max\limits_{1\leq k<n} \frac{1}{n^{3/2}}|| U_{n,k}||$ has already been established in Theorem \ref{theo1}, it is enough to prove the convergence in distribution of $\max\limits_{1\leq k<n} \frac{1}{n^{3/2}}|| U_{n,k}^\star||$ conditional on $X_1,...,X_n$. For this, we apply the Hoeffding decomposition:
\begin{multline*}
\frac{1}{n^{3/2}}U_{n,k}^\star=\frac{1}{n^{3/2}}\sum_{i=1}^k\sum_{j=k+1}^n h(X_i,X_j)(\varepsilon_i+\varepsilon_j)\\
=\frac{1}{n^{3/2}}\sum_{i=1}^k\sum_{j=k+1}^n (h_1(X_i)-h_1(X_j)(\varepsilon_i+\varepsilon_j)+\frac{1}{n^{3/2}}\sum_{i=1}^k\sum_{j=k+1}^n h_2(X_i,X_j)(\varepsilon_i+\varepsilon_j)
\end{multline*}
The second sum converges to 0 by Proposition \ref{degPart}. The first summand can be split into three parts with a short calculation:
\begin{multline*}
\frac{1}{n^{3/2}}\sum_{i=1}^k\sum_{j=k+1}^n (h_1(X_i)-h_1(X_j))(\varepsilon_i+\varepsilon_j)=\frac{1}{\sqrt{n}}\left(\sum_{i=1}^kh_1(X_i)\varepsilon_i+\frac{k}{n}\sum_{i=1}^nh_1(X_i)\varepsilon_i\right)\\
+\frac{1}{n^{3/2}}\sum_{i=1}^kh_1(X_i)\sum_{j=1}^n\varepsilon_j+\frac{1}{n^{3/2}}\sum_{i=1}^nh_1(X_i)\sum_{j=1}^k\varepsilon_j
\end{multline*}
By Proposition \ref{bootpartial} and the continuous mapping theorem, we have the weak convergence 
\begin{equation*}
\max_{1\leq k<n} \left\| \frac{1}{\sqrt{n}}\left(\sum_{i=1}^kh_1(X_i)\varepsilon_i+\frac{k}{n}\sum_{i=1}^nh_1(X_i)\varepsilon_i\right)\right\|\Rightarrow \sup_{\lambda\in[0,1]}\left\| W^\star(\lambda)-\lambda W^\star(1)\right\|
\end{equation*}
conditional on $X_1,...,X_n$. For the second part, note that 
\begin{equation*}
\Var\left(\frac{1}{n}\sum_{i=1}^n \varepsilon_i\right)=\frac{1}{n^2}\sum_{i,j=1}^nw(|i-j|/q_n)\leq \frac{1}{n}\sum_{i=-n}^n |w(i/q_n)|\approx \frac{q_n}{n}\int_{-\infty}^\infty |w(x)|dx\rightarrow0
\end{equation*}
for $n\rightarrow\infty$ by our assumptions on $q_n$. So $\frac{1}{n}\sum_{i=1}^n \varepsilon_i\rightarrow 0$ in probability and 
\begin{equation*}
\max_{k=1,...,n}\Big|\frac{1}{n^{3/2}}\sum_{i=1}^kh_1(X_i)\sum_{j=1}^n\varepsilon_j\Big|=\max_{k=1,...,n}\Big|\frac{1}{n^{1/2}}\sum_{i=1}^kh_1(X_i)\Big|\Big|\frac{1}{n}\sum_{j=1}^n\varepsilon_j\Big|\rightarrow0
\end{equation*}
for $n\rightarrow\infty$ in probability using the fact that $\frac{1}{n^{1/2}}\sum_{i=1}^kh_1(X_i)$ is stochastically bounded, see Proposition \ref{linThm}. For the third part, we consider increments of the partial sum and bound the variance of increments similar as above by
\begin{equation*}
\Var\left(\sum_{i=l+1}^k \varepsilon_i\right)\leq Ckq_n.
\end{equation*}
Because the $\varepsilon_i$ are Gaussian, it follows that
\begin{equation*}
\E\left[\Big(\sum_{i=l+1}^k\varepsilon_i\Big)^4\right]\leq C(kq_n)^2.
\end{equation*}
By Theorem 1 of \cite{M76}, we have
\begin{equation*}
\E\left[\max_{k=1,..,n}\Big(\sum_{i=1}^k\varepsilon_i\Big)^4\right]\leq C(nq_n)^2.
\end{equation*}
and $\frac{1}{n}\max_{k=1,..,n}|\sum_{i=1}^k\varepsilon_i|\rightarrow 0$ in probability because $q_n/n\rightarrow0$. So
\begin{equation*}
\max_{k=1,...,n}\Big|\frac{1}{n^{3/2}}\sum_{i=1}^nh_1(X_i)\sum_{j=1}^k\varepsilon_j\Big|=\Big|\frac{1}{n^{1/2}}\sum_{i=1}^nh_1(X_i)\Big|\max_{k=1,...,n}\Big|\frac{1}{n}\sum_{j=1}^k\varepsilon_j\Big|\xrightarrow{n\rightarrow\infty}0
\end{equation*}
which completes the proof.
\end{proof}

\subsection*{Appendix}

\begin{table}[ht]
\centering
\resizebox{\columnwidth}{!}{%
\begin{tabular}{ |c|c|c|c|c|c|c|c|c| } 
\hline
\multicolumn{9}{|c|}{Empirical Power}\\
\hline
& \multicolumn{2}{|c|}{Scenario 1} & \multicolumn{2}{|c|}{Scenario 2} & \multicolumn{2}{|c|}{Scenario 3} & \multicolumn{2}{|c|}{Scenario 4} \\
\hline
$\alpha$ &   CUSUM   & Spatial Sign &    CUSUM  & Spatial Sign & CUSUM     & Spatial Sign & CUSUM     & Spatial Sign \\
\hline
$0.1$    &   $0.907$ & $0.975$     & $0.774$   &  $0.903$   &  $0.635$ &  $0.981$ &  $0.038$& $0.994$   \\
$0.05$   &  $0.796$ & $0.929$    & $0.609$ &  $0.802$   &  $0.456$ &  $0.934$ & $0.014$ & $0.967$ \\
$0.025$  & $0.660$   & $0.846$    & $0.451$ &  $0.650$   &  $0.283$ &  $0.839$ & $0.005$ & $0.906$  \\
$0.01$   & $0.409$   & $0.627$    & $0.231$ &  $0.405$   &  $0.115$ &  $0.621$  & $0.001$   & $0.721$  \\
\hline
\end{tabular}%
}
\caption{Empirical power of CUSUM and patial sign for different significance level $\alpha$, Scenario 1-4.}
\label{T_A123}
\end{table}

The two additional scenarios to analyse what happens if the change point lies more closely to the beginning of the observations or if $d>>n$ are designed as follows:
\begin{itemize}
\item[Scenario 5:] Uniform Jump of $+0.3$ after $\gamma n$ observations:
\[ Y_i=\begin{cases} X_i \;\;\; & i<\gamma n \\ X_i +0.3 & i\geq \gamma n  \end{cases}  \text{  with $\gamma=0.3$ and $\gamma=0.15$ resp.}\]
\item[Scenario 6:] As Scenario 1 but with $n=150$, $d=350$.
\end{itemize}

\begin{table}[ht]
\centering
\begin{tabular}{|c|c|c|c|c|}
\hline
\multicolumn{3}{|c|}{Empirical Size - Scenario 6}\\
\hline
$\alpha$ &   CUSUM   & Spatial Sign \\
\hline
$0.1$    &  $0.067$& $0.064$    \\
$0.05$   &  $0.025$& $0.020$     \\
$0.025$  &  $0.005$& $0.007$    \\
$0.01$   &  $0.001$& $0.002$    \\
\hline
\end{tabular}
\caption{Empirical size of CUSUM and spatial sign for different significance level $\alpha$, Scenario 6.}
\label{T_SLd}
\end{table}

\begin{table}[ht]
\centering
\resizebox{\columnwidth}{!}{%
\begin{tabular}{ |c|c|c|c|c|c|c| } 
\hline
\multicolumn{7}{|c|}{Empirical Power}\\
\hline
& \multicolumn{2}{|c|}{Scenario 5, $\gamma=0.3$} & \multicolumn{2}{|c|}{Scenario 5, $\gamma=0.15$} & \multicolumn{2}{|c|}{Scenario 6} \\
\hline
$\alpha$ &   CUSUM   & Spatial Sign &    CUSUM  & Spatial Sign & CUSUM     & Spatial Sign \\
\hline
$0.1$    &  $0.795$&  $0.909$   &  $0.337$  &  $0.372$     &  $0.759$  & $0.903$    \\
$0.05$   &  $0.619$&  $0.783$   &  $0.173$&  $0.193$     &  $0.586$& $0.757$      \\
$0.025$  &  $0.439$  &  $0.599$   &  $0.076$  &  $0.084$   &  $0.391$  & $0.552$    \\
$0.01$   &  $0.218$  &  $0.329$   &  $0.024$&  $0.027$     &  $0.145$& $0.239$      \\
\hline
\end{tabular}%
}
\caption{Empirical power of CUSUM and spatial sign for different significance level $\alpha$, Scenario 5 and 6.}
\label{T_A56}
\end{table}

The size-power plots of Scenarios 5 and 6 (Figure \ref{S_P_A}) show that spatial sign based test suffers less loss in power than the CUSUM test if the change-point lies closer to the beginning of the observations or if $d$ becomes larger than $n$.

In particular we see (Table \ref{T_A56}) that in Scenario 5 with $\gamma=0.3$, the power of both statistics is smaller than in Scenario 1 where the change-point is in the middle of the observations. Nevertheless, the empirical power of the spatial sign based test is still larger than the empirical power of CUSUM and for $\alpha=0.1$ spatial signs still provides empirical power of about $0.9$. For $\gamma=0.15$ we see a drastic decline in power for both statistics, with empirical power smaller than $0.4$ even for $\alpha=0.1$. Spatial sign nevertheless keeps a small advantage over CUSUM in this scenario. 

In the last scenario we observe the situation of $d>>n$. For empirical size, we generated data as described in Chapter \ref{Simulations}, but with $n=150$ and $d=350$ and received the values presented in Table \ref{T_SLd}. We see that the size of both statistics is even smaller than under Scenario 1. But looking at the empirical power (Table \ref{T_A56}), we see a reduction of power for both statistics compared to Scenario 1. Nevertheless, we can still observe that the Wilcoxon-type test provides a greater empirical power than CUSUM. Particularly for $\alpha=0.1$, the test using spatial sign still shows a power of about $0.9$.

\begin{figure}
\begin{center}
\includegraphics[width=\textwidth]{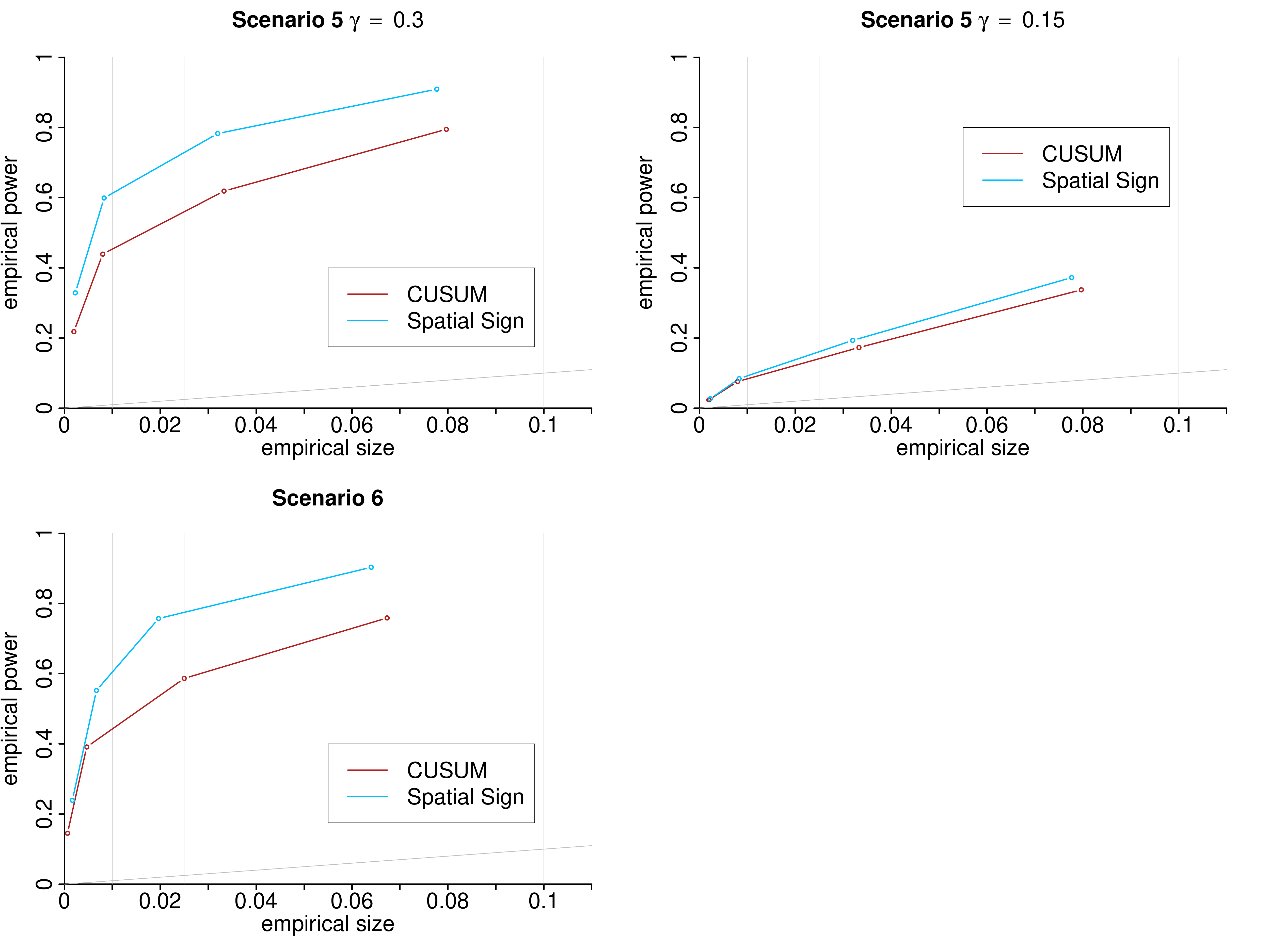}
\caption{Size-Power-Plot for CUSUM and spatial sign, Scenarios 5 and 6.}
\label{S_P_A}
\end{center}
\end{figure}

\bibliography{literatur} 
\end{document}